\documentclass[10pt,a4paper]{article}
\usepackage{amsmath}
\usepackage{amssymb}
\usepackage{amssymb}
\usepackage{cases}
\usepackage{bbm}
\usepackage{mathrsfs}
\usepackage{graphicx}
\usepackage{amsfonts}
\usepackage{enumerate}
\usepackage{theorem}
\usepackage[pdfstartview=FitH]{hyperref}
\makeatletter
\def\tank#1{\protected@xdef\@thanks{\@thanks
 \protect\footnotetext[0]{#1}}}
\def\bigfoot{

 \@footnotetext}
\makeatother

\newcommand{\ea}{\end{array}}

\allowdisplaybreaks

\newtheorem{theorem}{Theorem}[section]
\newtheorem{lem}{Lemma}[section]
\newtheorem{prp}[theorem]{Proposition}
\newtheorem{thm}[theorem]{Theorem}
\newtheorem{cor}[theorem]{Corollary}

\newtheorem{remark}{Remark}

\def\beq{\begin{equation}}
\def\nneq{\end{equation}}

\def\bthm{\begin{thm}}
\def\nthm{\end{thm}}

\def\blem{\begin{lem}}
\def\nlem{\end{lem}}
\def\bprf{\begin{proof}}

\def\bprop{\begin{prop}}
\def\nprop{\end{prop}}
\def\brmk{\begin{rem}}
\def\nrmk{\end{rem}}

\def\bexa{\begin{exa}}
\def\nexa{\end{exa}}
\def\bcor{\begin{cor}}
\def\ncor{\end{cor}}

\def\e{\varepsilon}

{\theorembodyfont{\rmfamily}
}

\title{Moderate deviations for stochastic models of two-dimensional second grade fluids driven by L\'evy noise
}

\footnotesize{\author{Wuting Zheng$^{1,}$\thanks{zwtzjr@mail.ustc.edu.cn},\ \ Jianliang Zhai $^{1,}$\thanks{zhaijl@ustc.edu.cn},\ \
 Tusheng Zhang$^{2,}$\thanks{Tusheng.Zhang@manchester.ac.uk}\\
 {\em $^1$ School of Mathematical Sciences,}\\
 {\em University of Science and Technology of China,}\\
 {\em Hefei, 230026, China}\\
 {\em $^2$ School of Mathematics, University of Manchester,}\\
 {\em Oxford Road, Manchester, M13 9PL, UK}\\
}
\date{}
\newenvironment{proof}{\par\noindent{\bf Proof:}}{\hspace*{\fill}$\blacksquare$\par}
\begin{document}
\maketitle

\noindent \textbf{Abstract:}
 In this paper, we establish a moderate deviation principle for stochastic models of two-dimensional second grade fluids driven by L\'evy noise. We will adopt the weak convergence approach. Because of the appearance of jumps, this result is significantly different from that in Gaussian case.
\vspace{4mm}

\numberwithin{equation}{section}
\section{Introduction}

The second grade fluids is an admissible
model of slow flow fluids,  which contains industrial fluids, slurries, polymer melts, etc.. It has attracted much attention
from a theoretical point of view, since it has properties of boundedness, stability and exponential decay, and has interesting connections with many other fluid models, see e.g. \cite{2003-Busuioc-p1119-1119}, \cite{1974-Dunn-p191-252}, \cite{1979-Fosdick-p145-152}, \cite{2001-Shkoller-p539-543} and references therein.
\vskip 0.2cm

Recently, taking into account the effect of random environment, the external force is considered as random. The stochastic models of two-dimensional second grade fluids have been studied. For the case of
Gaussian noises, we refer to \cite{CC,RS-10-01,RS-10,RS-12,WZZ,ZZ,ZZZ1}, where the authors obtained the existence and uniqueness of solutions,
 the behavior of the solutions as $\alpha\rightarrow0$,  Freidlin-Wentzell's large deviation principles (LDP), exponential mixing and moderate deviation principles (MDP) for the solutions. In the case
 of L\'evy noises, the global existence of a martingale solution was obtained in \cite{HRS}, the existence and uniqueness of strong probabilistic solutions is established in \cite{SZZ}, and the Freidlin-Wentzell's large deviation principles for the solutions is proved in \cite{ZZZ2}.

In this paper, we are concerned with asymptotic behaviors of stochastic models for the incompressible non-Newtonian fluids of second grade driven by L\'evy noise, which are given as follows:

\begin{align}\label{1.a}
\left\{
\begin{aligned}
& d(u^\e(t)-\alpha \Delta u^\e(t))+ \Big(-\kappa \Delta u^\e(t)+{\rm curl}(u^\e(t)-\alpha \Delta u^\e(t))\times u^\e(t)+\nabla\mathfrak{P}\Big)\,dt \\
& = F(u^\e(t),t)\,dt +\e\int_{\mathbb{Z}}G(u^\e(t-),z)\widetilde{N}^{\e^{-1}}(dzdt),\quad \rm{ in }\ \mathcal{O}\times(0,T], \\
&\begin{aligned}
& {\rm{div}}\,u^\e=0 \quad &&\rm{in}\ \mathcal{O}\times(0,T]; \\
& u^\e=0  &&\rm{in}\ \partial \mathcal{O}\times[0,T]; \\
& u^\e(0)=u_0  &&\rm{in}\ \mathcal{O},
&\end{aligned}
\end{aligned}
\right.
\end{align}

\noindent where $\mathcal{O}$ is a bounded open domain of $\mathbb{R}^2$; $u^\e=(u^\e_1,u^\e_2)$ and $\mathfrak{P}$ represent the random velocity and the modified pressure respectively. $\mathbb{Z}$ is
 a locally compact Polish space. On a specified complete filtered probability space $(\Omega,\mathcal{F},\{\mathcal{F}_t\}_{t\in[0,T]},P)$,
$\widetilde{N}^{\epsilon^{-1}}$ is a compensated Poisson random measure on $[0,T]\times\mathbb{Z}$ with a
$\sigma$-finite
mean measure $\epsilon^{-1}\lambda_T\otimes \nu$, where $\lambda_T$ is the Lebesgue measure on $[0,T]$ and $\nu$ is a $\sigma$-finite measure on $\mathbb{Z}$. The details of $(\Omega,\mathcal{F},\{\mathcal{F}_t\}_{t\in[0,T]},P, \widetilde{N}^{\epsilon^{-1}})$ will be given in Section 2.

\indent Let $\Pi$ be the Helmholtz-Leray projection from $\mathbb{L}^2(\mathcal{O})$ into $\mathbb{H}$. Let $A$ be the Stokes operator $-\Pi\Delta$(see the precise definition below). One can see that (\ref{1.a}) is equivalent to the following stochastic evolution equation:
\begin{eqnarray}\label{Abstract}
du^{\e}(t)=-\kappa\widehat{A}u^{\e}(t)dt-\widehat{B}(u^{\e}(t),u^{\e}(t))dt+\widehat{F}(u^{\e}(t),t)dt
+\e\int_{\mathbb{Z}}\widehat{G}(u^{\e}(t-),z)\widetilde{N}^{\e^{-1}}(dzdt),
\end{eqnarray}
with initial value $u^{\e}(0)=u_0$,\\
where $\widehat{A}=(I+\alpha A)^{-1}A$, $\widehat{B}(u,v)=(I+\alpha A)^{-1}\Big(curl(u-\alpha\Delta u)\times v\Big)$, $\widehat{F}=(I+\alpha A)^{-1}F$, $\widehat{G}=(I+\alpha A)^{-1}G$.

\vskip 0.3cm
As the parameter $\varepsilon$ tends to zero, the solution $u^\e$ of (\ref{Abstract}) will tend to the solution of the following deterministic equation
\begin{eqnarray}\label{Deterministic}
du^0(t)=-\kappa\widehat{A}u^0(t)-\widehat{B}(u^0(t),u^0(t))dt+\widehat{F}(u^0(t),t)dt,
\end{eqnarray}
with initial value $u^0(0)=u_0$.

\vskip 0.2cm
In this paper, we shall investigate deviations of $u^{\e}$ from the deterministic solution $u^0$, as $\e$ decreases to $0$, that is, the asymptotic behavior of the trajectory,
$$Z^{\e}(t)=\frac{1}{a(\e)}(u^{\e}-u^0)(t),\ \ t\in[0,T],$$
where $a(\e)$ is some deviation scale which strongly influences the asymptotic behavior of $Z^\e$.
We will study the so-called moderate deviation principle (MDP for short), that is when the deviation scale satisfies
\begin{equation}\label{condition}
a(\e)\rightarrow0,\ \ \ \ {\e}/a^2(\e)\rightarrow 0 \ \ \ \ \text{as}\ \ \e\rightarrow 0.
\end{equation}
Throughout this paper, we assume that (\ref{condition}) is in place.\\
\indent Like the large deviations, the estimates of moderate deviations are very useful in the theory of statistical inference. It can provide us with the rate of convergence and a useful method for constructing asymptotic confidence intervals, see \cite{Ermakov,Gao,Inglot,Kallenberg} and references therein. There are many methods to establish the MDP in various framework, for example, De Acosta \cite{Acosta}, Chen\cite{Chen} and Ledoux \cite{Ledoux} for processes with independent increments; Wu \cite{Wu} for Markov processes; Guillin and Liptser \cite{Guillin} for diffusion processes; Wang and Zhang \cite{Wang2} for stochastic reaction-diffusion equations; Wang, Zhai and Zhang \cite{Wang1} for 2-D stochastic Navier-Stokes equations driven by Brownian motion; Zhai and Zhang \cite{ZZ} for stochastic models of 2-D second grade fluids driven by Brownian motion.\\
\indent The MDP for stochastic evolution equation and stochastic partial differential equations driven by L\'evy noise are quite different from that in the case driven by Brownian motion because of the difficulties caused by the jumps. In this paper, we will adopt the weak convergence approach introduced in \cite{BDG2} to establish the MDP for stochastic models of second grade fluids driven by L\'evy noise. Similar to \cite{DXZZ}, we decompose the solutions into a sum of the solutions of several relatively simpler equations and prove the convergence/tightness of the solutions of each equations. But the details of the proof are quite different and more difficult, because of the nature of the second grade fluids models. The main effort is to deal with the nonlinear term $\rm{curl}(u^\e(t)-\alpha \Delta u^\e(t))\times u^\e(t)$.\\
\indent We organize this paper as follows. In Section 2, we introduce some functional spaces and some notations. In Section 3, we formulate the hypotheses and state our main result. In Section 4, we provide all the proofs.

\section{Preliminaries and Notations}

In this paper, we assume that $\mathcal{O}$ is a simply connected and bounded open domain of $\mathbb{R}^2$ with boundary $\partial \mathcal{O}$ of class $\mathcal{C}^{3,1}$. For $p\geq 1$ and $k\in\mathbb{N}$, we denote by $L^p(\mathcal{O})$
and $W^{k,2}(\mathcal{O})$ the usual $L^p$ and Sobolev spaces over $\mathcal{O}$ respectively.
Let $W^{k,2}_0(\mathcal{O})$ be
the closure in $W^{k,2}(\mathcal{O})$ of $\mathcal{C}^\infty_c(\mathcal{O})$ the space of infinitely differentiable functions with compact supports in $\mathcal{O}$. For simplicity, we write $H^k(\mathcal{O}):=W^{k,2}(\mathcal{O})$ and $H_0^k(\mathcal{O}):=W^{k,2}_0(\mathcal{O})$. We equip $H^1_0(\mathcal{O})$
with the scalar product
\begin{align*}
((u,v))=\int_\mathcal{O}\nabla u\cdot\nabla vdx=\sum_{i=1}^2\int_\mathcal{O}\frac{\partial u}{\partial x_i}\frac{\partial v}{\partial x_i}dx,
\end{align*}

\noindent where $\nabla$ is the gradient operator. It is well known that the norm $\|\cdot\|$ generated by this scalar product is equivalent to the usual norm of $H^1(\mathcal{O})$.
%If the domain is smooth enough, then for any $m$ and $p$ such that $mp>n$ the embedding $W^{j+m,p}(\mathcal{O})\subset W^{j,q}(\mathcal{O})$ is compact for any $1\leq q\leq\infty$.
%More general embedding theorems can be found in \cite{2003-Adams-p-} and reference therein.

Throughout this paper, we set $\mathbb{Y}=Y\times Y$ for any Banach space $Y$.
Set
\begin{align*}\label{SP-01}
&\mathcal{C}=\Big\{u\in[\mathcal{C}^\infty_c(\mathcal{O})]^2\ {\rm such \ that\  div}\ u=0\Big\},\nonumber\\
&\mathbb{H}={\rm\ closure\ of}\ \mathcal{C}\ {\rm in}\ \mathbb{L}^2(\mathcal{O})(:=L^2(\mathcal{O},\mathbb{R}^2)),\nonumber\\
&\mathbb{V}={\rm\ the\ closure\ of}\ \mathcal{C} {\rm\ in}\ \mathbb{H}^1(\mathcal{O}).\nonumber
\end{align*}

We denote by $(\cdot,\cdot)$ and $|\cdot|$ the inner product in $\mathbb{L}^2(\mathcal{O})$( in $\mathbb{H}$) and the induced norm, respectively. The inner product and the norm of $\mathbb{H}^1_0(\mathcal{O})$
are denoted respectively by $((\cdot,\cdot))$ and $\|\cdot\|$. We endow the space $\mathbb{V}$ with the norm generated by the following inner product
\[
(u,v)_\mathbb{V}:=(u,v)+\alpha ((u,v)),\quad \text{for any } u,v\in\mathbb{V},
\]

\noindent and the norm in $\mathbb{V}$ is denoted by $\|\cdot\|_{\mathbb{V}}$. The $\rm Poincar\acute{e}$'s inequality implies that there exists a constant $\mathcal{P}>0$ such that the following inequalities holds
\begin{align}\label{Poincare}
(\mathcal{P}^2+\alpha)^{-1}\|v\|^2_\mathbb{V} \leq \|v\|^2
\leq\alpha^{-1}\|v\|^2_\mathbb{V},\quad \text{for any } v\in\mathbb{V}.
\end{align}
%where $\mathcal{P}$ is the constant from Poincar\'e's inequality.

\vskip 0.2cm
We also introduce the following space
\[
\mathbb{W}=\big\{u\in\mathbb{V}: {\rm curl}(u-\alpha\Delta u)\in L^2(\mathcal{O})\big\},
\]
and endow it with the norm generated by the scalar product
\begin{align}\label{W}
(u,v)_\mathbb{W}:=\big({\rm curl}(u-\alpha\Delta u),{\rm curl}(v-\alpha\Delta v)\big).
\end{align}
The norm in $\mathbb{W}$ is denoted by $\|\cdot\|_{\mathbb{W}}$. It has been proved that, see e.g. \cite{CG,CO}, the following (algebraic and topological) identity holds:
\begin{align*}
\mathbb{W}=\big\{v\in\mathbb{H}^3(\mathcal{O}): {\rm div}\,v=0\ {\rm and}\  v|_{\partial \mathcal{O}}=0\big\},
\end{align*}
moreover, there exists a constant $C> 0$ such that
\begin{align}\label{W-02}
    |v|_{\mathbb{H}^3(\mathcal{O})}
\leq
    C\|v\|_\mathbb{W},\ \ \ \forall v\in \mathbb{W}.
\end{align}
This result states that the norm $\|\cdot\|_\mathbb{W}$ is equivalent to the usual norm in $\mathbb{H}^3(\mathcal{O})$.

Identifying the Hilbert space $\mathbb{V}$ with its dual space $\mathbb{V}^*$ by the Riesz representation, we get a
Gelfand triple
\begin{align*}
\mathbb{W}\subset \mathbb{V}\subset\mathbb{W}^*.
\end{align*}

\noindent We denote by $\langle f,v\rangle$ the dual relation between $f\in\mathbb{W}^*$ and $v\in\mathbb{W}$ from now on. It is easy to see
\begin{eqnarray}\label{basis}
(v,w)_\mathbb{V}=\langle v,w\rangle,\ \ \ \forall\,v\in\mathbb{V},\ \ \forall\,w\in\mathbb{W}.
\end{eqnarray}

Note that the injection of $\mathbb{W}$ into $\mathbb{V}$ is compact,
thus there exists a sequence $\{e_i\}$ of elements of $\mathbb{W}$ which forms an orthonormal basis in $\mathbb{W}$, and an orthogonal system in $\mathbb{V}$, moreover this sequence verifies:
\begin{align}\label{Basis}
(v,e_i)_{\mathbb{W}}=\lambda_i(v,e_i)_{\mathbb{V}},\ \text{for any }v\in\mathbb{W},
\end{align}
where $0<\lambda_i\uparrow\infty$. From Lemma 4.1 in \cite{CG} we have
\begin{align}\label{regularity of basis}
e_i\in \mathbb{H}^4(\mathcal{O}),\ \ \forall\,i\in\mathbb{N}.
\end{align}

Consider the following ``generalized Stokes equations'':
\begin{align}\label{General Stokes}
\begin{aligned}
v-\alpha \Delta v &=f\quad {\rm in}\quad\mathcal{O},\\
{\rm div}\,v &=0\quad {\rm in}\quad\mathcal{O},\\
v &=0\quad {\rm on}\quad\partial \mathcal{O}.
\end{aligned}
\end{align}

\noindent The following result can be derived from \cite{Solonnikov} and also can be found in \cite{RS-10,RS-12}.
\begin{lem}\label{Lem GS}
%Let $\mathcal{O}$ be a connected, bounded open subset of $\mathbb{R}^2$ with a boundary $\partial \mathcal{O}$ of class $\mathcal{C}^l$
Set $l=1,2,3$. Let $f$ be a function in $\mathbb{H}^l$, then the system (\ref{General Stokes}) has a unique solution $v$. Moreover if $f$ is an element of $\mathbb{H}^l\cap\mathbb{V}$, then $v\in \mathbb{H}^{l+2}\cap\mathbb{V}$, and the following relations hold
\begin{gather}
(v,g)_\mathbb{V}=(f,g),\quad \forall\, g\in \mathbb{V},\label{Eq GS-01}\\
|v|_{\mathbb{H}^{l+2}}\leq C|f|_{\mathbb{H}^l}.\label{Eq GS-02}
\end{gather}
\end{lem}

We recall the following estimates which can be found in \cite{RS-12}.
\begin{lem}\label{Lem B}
For any $u,v,w\in\mathbb{W}$, we have
\begin{align}\label{Ineq B 01}
    |({\rm curl}(u-\alpha\Delta u)\times v,w)|
\leq
    C\|u\|_{\mathbb{W}}\|v\|_\mathbb{V}\|w\|_{\mathbb{W}},
\end{align}

\noindent and
\begin{align}\label{Ineq B 02}
    |({\rm curl}(u-\alpha\Delta u)\times u,w)|
\leq
    C\|u\|^2_\mathbb{V}\|w\|_{\mathbb{W}}.
\end{align}
\end{lem}

Defining the bilinear operator $\widehat{B}(\cdot,\cdot):\ \mathbb{W}\times\mathbb{V}\longrightarrow\mathbb{W}^*$ by
\begin{align*}
\widehat{B}(u,v):=(I+\alpha A)^{-1}\mathbb{P}\big({\rm curl}(u-\alpha \Delta u)\times v\big).
\end{align*}
%The short notation $\widehat{B}(u):=\widehat{B}(u,u)$ is often used.
We have the following consequence of Lemma \ref{Lem B}.

\begin{lem}\label{Lem-B-01}

For any $u\in\mathbb{W}$ and $v\in\mathbb{V}$, it holds that
\begin{align}\label{Eq B-01}
    \|\widehat{B}(u,v)\|_{\mathbb{W}^*}
\leq
    C\|u\|_\mathbb{W}\|v\|_\mathbb{V},
\end{align}

\noindent and
\begin{align}\label{Eq B-02}
\|\widehat{B}(u,u)\|_{\mathbb{W}^*}
\leq
    C\|u\|^2_\mathbb{V}.
\end{align}

\noindent In addition
\begin{align}\label{Eq B-03}
 \langle\widehat{B}(u,v),v\rangle=0, \quad\forall\,u, v\in\mathbb{W},
\end{align}

\noindent which implies
\begin{align}\label{Eq B-04}
 \langle\widehat{B}(u,v),w\rangle=-\langle\widehat{B}(u,w),v\rangle,\quad\forall\,u, v, w\in\mathbb{W}.
\end{align}

\end{lem}

We are now introducing $(\Omega,\mathcal{F},\mathbb{F}:=\{\mathcal{F}_t\}_{t\in[0,T]},P, \widetilde{N}^{\epsilon^{-1}})$.
\vskip 0.2cm

\indent For a locally compact Polish space $S$, let $M_{FC}(S)$ denote the space of all Borel measures $\vartheta$ on $S$ such that  $\vartheta(K)<\infty$
for each compact set $K\subseteq S$. Endow $M_{FC}(S)$ with the weakest topology, denoted it by $\mathcal{T}(M_{FC}(S))$, such that for each $f\in C_c(S)$ the mapping
$\vartheta\in M_{FC}(S)\rightarrow \int_Sf(s)\vartheta(ds)$ is continuous. This topology is metrizable such that $M_{FC}(S)$ is a Polish space, see \cite{BDM} for more details.
%Fix $T\in(0,\infty)$ and let $\mathBB{Z}_{T}=[0,T]\times \mathBB{Z}$. Fix a measure $\nu\in M_{FC}(S)$, and let $\nu_{T}=\lambda_T\otimes\nu$, where $\lambda_T$ is Lebesgue measure on $[0,T]$.
\vskip 0.2cm

\indent Recall that $\mathbb{Z}$ is a locally compact Polish space, and in this paper, we assume that $\nu$ is a given element of $M_{FC}(\mathbb{Z})$.  We specify the underlying probability space $(\Omega, \mathcal{F}, {\mathbb{F}}:=\{\mathcal{F}_t\}_{t\in [0,T]},P)$ in the following way:
\begin{align*}
  \Omega:=M_{FC}\big([0,T]\times \mathbb{Z}
  \times [0,\infty)\big),
  \qquad \mathcal{F}:=\mathcal{T}(M_{FC}\big([0,T]\times \mathbb{Z}\times [0,\infty)).
\end{align*}
\indent We introduce the function
\begin{align*}
& N\colon \Omega \rightarrow M_{FC}\big([0,T]\times \mathbb{Z}\times [0,\infty)\big),\qquad  N(\omega)=\omega.
\end{align*}
%Define
%$$\bar{\mathbb{W}}=C\big([0,T],R\big)\times M_{FC}\big([0,T]\times \mathBB{Z}\times [0,\infty)\big).$$
\indent Define for each $t\in [0,T]$ the $\sigma$-algebra
\begin{align*}
\mathcal{G}_{t}:=\sigma\left(\left\{N((0,s]\times A)\big):\,
0\leq s\leq t,\,A\in \mathcal{B}\big(\mathbb{Z}\times [0,\infty)\big)\right\}\right).
\end{align*}
\indent Let $\lambda_T$ and $\lambda_\infty$ be Lebesgue measure on $[0,T]$ and $[0,\infty)$ respectively. It follows from \cite[Sec.I.8]{Ikeda-Watanabe} that there exists a unique probability measure $P$
 on $(\Omega,\mathcal{F})$ such that: $N$ is a Poisson random measure on $\Omega$ with intensity measure $\lambda_T\otimes\nu\otimes \lambda_\infty$.\\
\indent We denote by $\mathbb{F}:=\{{\mathcal{F}}_{t}\}_{t\in[0,T]}$ the $P$-completion of $\{\mathcal{G}_{t}\}_{t\in[0,T]}$ and by $\mathcal P$ the $\mathbb{F}$-predictable $\sigma$-field on $[0,T]\times \Omega$. Define
\begin{align*}
{\mathcal{A}}
:=\left\{\varphi\colon [0,T]\times {\mathbb{Z}}\times\Omega\to [0,\infty):
\, (\mathcal{P}\otimes\mathcal{B}({\mathbb{Z}}))\setminus\mathcal{B}[0,\infty)\text{-measurable}\right\}.
\end{align*}

For $\varphi\in{\mathcal{A}}$, define a
counting process $N^{\varphi}$ on $[0,T]\times {{\mathbb{Z}}}$ by
   \begin{align*}%\label{Jump-representation}
      N^\varphi((0,t]\times A)=\int_{(0,t]\times A\times (0,\infty)}1_{[0,\varphi(s,z)]}(r)\, N(ds, dz, dr),
   \end{align*}
for $t\in[0,T]$ and $A\in\mathcal{B}({\mathbb{Z}})$. When $\varphi(s,z,\omega)=\epsilon^{-1}$, we write $N^\varphi=N^{\epsilon^{-1}}$.
It is easy to see that $N^{\epsilon^{-1}}$ is a Poisson random measure on $[0,T]\times\mathbb{Z}$ with a mean measure $\epsilon^{-1}\lambda_T\otimes\nu$. We denote $\widetilde{N}^{\epsilon^{-1}}$ the compensated Poisson random measure respect to $N^{\epsilon^{-1}}$.

\vskip 0.2cm

We end this section with a criteria of compactness, which will be used later. Let $\mathbb{K}$ be a separable Hilbert space. Given $p>1$, $\beta\in(0,1)$, let $W^{\beta,p}([0,T],\mathbb{K})$ be the
 space of all $u\in L^p([0,T],\mathbb{K})$ such that
$$
\int_0^T\int_0^T\frac{\|u(t)-u(s)\|^p_\mathbb{K}}{|t-s|^{1+\beta p}}dtds<\infty,
$$
endowed with the norm
$$
\|u\|^p_{W^{\beta,p}([0,T],\mathbb{K})}:=\int_0^T\|u(t)\|^p_{\mathbb{K}}dt+\int_0^T\int_0^T\frac{\|u(t)-u(s)\|^p_\mathbb{K}}{|t-s|^{1+\beta p}}dtds.
$$
The following result is a variant of the criteria for compactness proved in \cite{Lions} (Sect. 5, Ch. I)
 and \cite{Temam 1983} (Sect. 13.3).
\begin{lem}\label{Compact}{\rm
Let $\mathbb{K}_0\subset \mathbb{K}\subset \mathbb{K}_1$ be Banach spaces, $\mathbb{K}_0$ and $\mathbb{K}_1$ reflexive, with compact embedding of $\mathbb{K}_0$ into $\mathbb{K}$.
For $p\in(1,\infty)$ and $\beta\in(0,1)$, let $\Lambda$ be the space
$$
\Lambda=L^p([0,T];\mathbb{K}_0)\cap W^{\beta,p}([0,T];\mathbb{K}_1)
$$
endowed with the natural norm. Then the embedding of $\Lambda$ into $L^p([0,T];\mathbb{K})$ is compact.
}\end{lem}

\section{Hypothesis and Main Result}

In this section, we will state the precise assumptions on the coefficients and our main result.\\
\vskip 0.02cm
%We endow the complete probability space $(\Omega,\mathcal{F},P)$ with a filtration $\mathcal{F}_t$, $t\in[0,T]$.
Let $F:\mathbb{V}\times[0,T]\rightarrow\mathbb{V}$
and $G:\mathbb{V}\times\mathbb{Z}\rightarrow\mathbb{V}$ be given measurable maps. We introduce the following conditions:

\vskip 0.2cm
{\bf (F1)}
\begin{eqnarray}\label{F-01}
            F(0,t)=0,
\end{eqnarray}
and
\begin{eqnarray}\label{F-02}
\|F(u_1,t)-F(u_2,t)\|_\mathbb{V}
\leq
C_1\|u_1-u_2\|_\mathbb{V},\ \ \forall u_1,u_2\in\mathbb{V},\ \ t\in[0,T].
\end{eqnarray}

{\bf (F2)} $F$ is differentiable with respect to the first variable, and the derivative $F':\mathbb{V}\times[0,T] \rightarrow L(\mathbb{V})$ ( $L(\mathbb{V})$ is the space of all bounded linear operators from $\mathbb{V}$ to $\mathbb{V}$) is uniformly Lipschitz with respect to the first variable, more precisely,
\begin{eqnarray}\label{F-03}
\|F'(u_1,t)-F'(u_2,t)\|_{L(\mathbb{V})}
\leq
C\|u_1-u_2\|_\mathbb{V},\ \ \forall u_1,u_2\in\mathbb{V},\ \ t\in[0,T].
\end{eqnarray}
By (\ref{F-02}), we conclude that
\begin{eqnarray}\label{F-04}
   \|F'(u,t)\|_{L(\mathbb{V})}\leq C.
\end{eqnarray}
Denote $\widehat{F}'(u,t)=(I+\alpha A)^{-1} F'(u,t)$.

{\bf (G)}
There exist $L_G, M_G\in L^2(\nu)\cap\mathcal{H}$ such that
\begin{eqnarray}\label{G-01}
\|G(x_1,z)-G(x_2,z)\|_\mathbb{V}
\leq
L_G(z)\|x_1-x_2\|_\mathbb{V},\ \ \forall x_1,x_2\in\mathbb{V},\ \ z\in\mathbb{Z},
\end{eqnarray}
and
\begin{eqnarray}\label{G-02}
\|G(x,z)\|_\mathbb{V}
\leq
M_G(z)\big(1+\|x\|_\mathbb{V}\big),\ \ \forall x\in\mathbb{V},\ \ z\in\mathbb{Z},
\end{eqnarray}
where $\mathcal{H}=\Big\{h:\mathbb{Z}\rightarrow\mathbb{R}:\exists\delta>0,\text{ s.t. }\forall\Gamma \text{ with }\nu(\Gamma)<\infty,\int_{\Gamma}\exp{(\delta h^2(z))}\nu(dz)<\infty\Big\}$.
%We now define  three operators $\widehat{F}$ , $widehat{F}'$ and $\widehat{G}$ which map $\mathbb{V}\times[0,T]$ into $\mathbb{W}$ $L(\mathbb{V},\mathbb{W})$ and $\mathbb{W}^{\otimes m}$, respectively, by

Recall the solution $u^0$ given in (\ref{Deterministic}). By Theorem 5.6 in \cite{CG}, we have the following result.
\begin{lem}\label{Regularity}
If we assume that the boundary $\partial \mathcal{O}$ is of class $\mathcal{C}^{3,1}$ and the initial value $u_0\in{\mathbb{W}\cap{\mathbb{H}}^4(\mathcal{O})}$, then $u^0$ belongs to $L^\infty([0,T],{\mathbb{W}\cap{\mathbb{H}}^4(\mathcal{O})})$, i.e.
\begin{eqnarray}\label{Regularity estimate}
\sup_{t\in[0,T]}\|u^0(t)\|_{{\mathbb{H}}^4(\mathcal{O})}\leq C.
\end{eqnarray}
\end{lem}
%consider the condition in theorem, we should have $F(u,t)\in\mathbb{V}$ and $curl(F(u,t))\in\mathbb{V}$, since the theorem \ref{Solution Existence}, we can have $u^0(t)\in\mathbb{W}$, obviously, $F(u,t)\in\mathbb{V}$. Since the Condition (F1) , by simple calculus and $u^0(t)\in\mathbb{W}$, we can prove $curl(F(u,t))\in\mathbb{V}$.

\vskip 0.02cm
To obtain the moderate deviation principle, additionally we impose the following hypothesis throughout the paper:

\vskip 0.02cm
{\bf (I)} the initial value $u_0\in{\mathbb{W}\cap{\mathbb{H}}^4(\mathcal{O})}$.
%and the boundary $\partial \mathcal{O}$ is of class $\mathcal{C}^{3,1}$.

\vskip 0.2cm
\indent In order to introduce our main result, we need the following notations. The space $D([0,T],\mathbb{V})$ is the collection of all $\mathbb{V}$-valued c$\grave{a}$dl$\grave{a}$g functions equipped with the Skorokhod topology. For any $\e>0$ and $M<\infty$, consider the spaces
\begin{eqnarray*}
\mathcal{S}_{+,\e}^M=\big\{\varphi:[0,T]\times\mathbb{Z}\rightarrow\mathbb{R}_+|L_T(\varphi)\leq Ma^2(\e)\big\},\\
\mathcal{S}_{\e}^M=\big\{\psi:[0,T]\times\mathbb{Z}\rightarrow\mathbb{R}|\psi=(\varphi-1)/a(\e),\varphi\in\mathcal{S}_{+,\e}^M\big\},
\end{eqnarray*}
where $L_T(g)=\int_0^T\int_{\mathbb{Z}}\big(g(t,z)\log{g(t,z)}-g(t,z)+1\big)\nu(dz)dt$.\\
\indent The norm in the Hilbert space $L^2(\nu_T)$ will be denoted by $\|\cdot\|_2$ and $B_2(R)$ denotes the ball of radius $R$ in $L^2(\nu_T)$. Throughout this paper $B_2(R)$ is equipped with the weak topology of $L^2(\nu_T)$ and it is therefore weakly compact.\\

By Theorem 3.2 in Shang, Zhai and Zhang \cite{SZZ}, we know that the equation (\ref{Abstract}) has a unique strong solution $u^{\e}\in D([0,T],\mathbb{V})\cap L^{\infty}(0,T;\mathbb{W})$ in the probabilistic sense.
Set $Y^\e={\big(u^\e-u^0\big)}/{a(\e)}$, which satisfies
\begin{eqnarray}\label{MDP-eq}
dY^{\e}(t)
&=&
-\kappa\widehat{A}Y^{\e}(t)dt-
%\frac{1}{a(\e)}\Big(\widehat{B}(u^{\e}(t),u^{\e}(t))-\widehat{B}(u^0(t),u^0(t))\Big)dt
\Big(\widehat{B}(a(\e)Y^\e(t)+u^0(t),Y^{\e}(t))+\widehat{B}(Y^\e(t),u^0(t))\Big)dt\nonumber\\
& &+\frac{1}{a(\e)}\Big(\widehat{F}(a(\e)Y^\e(t)+u^0(t),t)-\widehat{F}(u^0(t),t)\Big)dt\nonumber\\
& &+\frac{\e}{a(\e)}\int_{\mathbb{Z}}\widehat{G}(a(\e)Y^\e(t-)+u^0(t-),z)\widetilde{N}^{\e^{-1}}(dzdt),
\end{eqnarray}
with initial value $Y^{\e}(0)=0$.\\
\indent The following theorem is our main result.
\begin{thm}\label{MDP Theorem}
Suppose that \textbf{Conditions (F1), (F2), (G) and (I)} hold. Then $\{Y^\e\}$ satisfies a large deviation principle in $D([0,T],\mathbb{V})$ with speed $\e/a^2(\e)$ and the rate function given by
$$I(\eta)=\inf_{\psi}[\frac{1}{2}\|\psi\|_2^2],$$
where the infimum is taken over all $\psi\in L^2(\nu_T)$ such that $(\eta,\psi)$ satisfies the following equation:
\begin{eqnarray}\label{Theorem-1-eq-1}
d\eta(t)
&=&
-\kappa\widehat{A}\eta(t)dt-
%\frac{1}{a(\e)}\Big(\widehat{B}(u^{\e}(t),u^{\e}(t))-\widehat{B}(u^0(t),u^0(t))\Big)dt
\Big(\widehat{B}(\eta(t),u^0(t))+\widehat{B}(u^0(t),\eta(t))\Big)dt\nonumber\\
& &+\widehat{F}'(u^{0}(t),t)\eta(t)dt+\int_{\mathbb{Z}}\widehat{G}(u^{0}(t),z)\psi(z,t)\nu(dz)dt,
\end{eqnarray}
with initial value $\eta(0)=0$.
That is,
\begin{itemize}
         \item[$(a)$](Upper bound) For each closed subset $O_1$ of $D([0,T],\mathbb{V})$,
              $$
                \limsup_{\e\rightarrow 0}\frac{1}{\e/a^2(\e)}\log P(Y^\e\in O_1)\leq- \inf_{x\in O_1}I(x).
              $$
         \item[$(b)$](Lower bound) For each open subset $O_2$ of $D([0,T],\mathbb{V})$,
              $$
                \liminf_{\e\rightarrow 0}\frac{1}{\e/a^2(\e)}\log P(Y^\e\in O_2)\geq- \inf_{x\in O_2}I(x).
              $$
       \end{itemize}
\end{thm}
\begin{remark}
Following the similar arguments as in the proof of  Theorem 5.6 in \cite{CG}, one can see that for all $\psi\in L^2(\nu_T)$, the equation (\ref{Theorem-1-eq-1}) has a unique solution $\eta\in L^\infty([0,T],{\mathbb{W}\cap{\mathbb{H}}^4(\mathcal{O})})$.
\end{remark}
\begin{proof}
Define $\mathcal{G}^0:L^2(\nu_T)\rightarrow C([0,T],\mathbb{V})$ by
\begin{eqnarray}
\mathcal{G}^0(\psi)=\eta \text{ for } \psi\in L^2(\nu_T), \text{ where } (\eta,\psi) \text{ solves } (\ref{Theorem-1-eq-1}).
\end{eqnarray}
\indent  The existence and uniqueness of the strong solution of (\ref{MDP-eq}) implies that there exists a measurable mapping $\mathcal{G}^{\e}:M_{FC}(\mathbb{Z}\times[0,T])\rightarrow D([0,T],\mathbb{V})$ such that: $\mathcal{G}^{\e}(\e N^{\e^{-1}})=Y^{\e}$.\\

\indent We will apply the general criteria (Theorem 2.3) obtained in \cite{BDG} to prove the theorem. According to \cite{BDG}, it is sufficient to verify two claims. The first one is the following:\\
\indent \textbf{(MDP-1)} For any $M>0$, suppose that $g^\e,g\in B_2(M)$ and $g^\e\rightarrow g$. Then
\begin{eqnarray*}
\mathcal{G}^0(g^\e)\rightarrow\mathcal{G}^0(g)\ \ \text{  in }C([0,T],\mathbb{V}).
\end{eqnarray*}
\indent In order to state the second claim, we need to introduce some additional notations.\\
\indent Let $(K_n)_{n\in\mathbb{N}}$ be a sequence of compact sets $K_n\subseteq \mathbb{Z}$
with $ K_n \nearrow \mathbb{Z}$.  For each $n\in\mathbb{N}$, let
\begin{align*}
     \bar{\mathcal{A}}_{b,n}
= \Big\{\psi\in \mathcal{A}: \ \
&\psi(t,z,\omega)\in[\tfrac{1}{n},n], \text{if }(t,z,\omega)\in [0,T]\times K_n\times{\Omega}\\
&\text{and }\psi(t,z,\omega)=1, \text{if }(t,z,\omega)\in [0,T]\times K_n^c\times{\Omega}
\Big\}
\end{align*}
and $\bar{\mathcal{A}}_{b}=\bigcup _{n=1}^\infty \bar{\mathcal{A}}_{b,n}$. Define
\begin{eqnarray*}
\mathcal{U}_{+,\e}^M=\big\{\varphi\in\bar{\mathcal{A}}_b:\varphi(\cdot,\cdot,\omega)\in\mathcal{S}_{+,\e}^M,{P}-a.s.\big\}\\
\mathcal{U}_{\e}^M=\big\{\psi\in\bar{\mathcal{A}}:\psi(\cdot,\cdot,\omega)\in\mathcal{S}_{\e}^M,P-a.s.\big\}
\end{eqnarray*}
\indent Suppose $\varphi\in\mathcal{S}_{+,\e}^M$. By Lemma 3.2 in \cite{BDG}, there exists $\kappa_2(1)\in(0,\infty)$ that is independent of $\e$ and such that $\psi1_{\{|\psi|\leq1/a(\e)\}}\in B_2(\sqrt{M\kappa_2(1)})$, where $\psi=(\varphi-1)/a(\e)$. In this paper, we use the symbol $"\Rightarrow"$ to denote convergence in distribution. Now we state the second claim:\\
\indent \textbf{(MDP-2)} For any $M>0$, let $\{\varphi^\e\}$ be such that for every $\e>0,\varphi^\e\in\mathcal{U}_{+,\e}^M$ and for some $\beta\in(0,1],\psi^\e1_{\{|\psi^\e|\leq\beta/a(\e)\}}\Rightarrow\psi$ in $B_2(\sqrt{M\kappa_2(1)})$ where $\psi^\e=(\varphi^\e-1)/a(\e)$. Then
\begin{eqnarray*}
\mathcal{G}^{\e}(\e N^{{\e}^{-1}\varphi^\e})\Rightarrow\mathcal{G}^0(\psi)\ \ \text{  in }D([0,T],\mathbb{V}).
\end{eqnarray*}
\indent The proofs of (MDP-1), (MDP-2) is lengthy and involved, we will give the details in the next section. \textbf{(MDP-1)} will be proved in Proposition \ref{MDP-1} and \textbf{(MDP-2)} will be established in Proposition \ref{MDP-2}.
\end{proof}

\section{The proofs of MDP-1 and MDP-2}

We need some more preparations before the proof. The following Lemmas \ref{H-lemma-1}-\ref{H-lemma-3} were proved in \cite{BDG}( see Lemma 4.2, Lemma 4.3 and Lemma 4.6 there).

\begin{lem}\label{H-lemma-1}
Let $h\in L^2(\nu)\cap\mathcal{H}$ and fix $M>0$. Then there exists a constant $\varsigma_h>0$ such that for any measurable subset $I\in[0,T]$ and for all $\e>0$,
\begin{eqnarray}\label{H-lemma-1-eq}
\sup_{\varphi\in{{\mathcal{S}}^M_{+,\e}}}\int_{\mathbb{Z}\times I}h^2(z)\varphi(z,s)\nu(dz)ds\leq\varsigma_h\big(a^2(\e)+\lambda_T(I)\big).
\end{eqnarray}
\end{lem}

\begin{lem}\label{H-lemma-2}
Let $h\in L^2(\nu)\cap\mathcal{H}$ and $I$ be a measurable subset of $[0,T]$. Fix $M>0$. Then there exists $\Gamma_h$, $\rho_h:(0,\infty)\rightarrow(0,\infty)$ such that $\Gamma_h(u)\downarrow0$ as $u\uparrow\infty$ and for all $\e,\beta\in(0,\infty)$,
\begin{eqnarray}\label{H-lemma-2-eq-1}
\sup_{\psi\in{{\mathcal{S}}^M_{\e}}}\int_{\mathbb{Z}\times I}|h(z)\psi(z,s)|1_{\{|\psi|>\beta/a(\e)\}}\nu(dz)ds\leq\Gamma_h(\beta)\big(1+\sqrt{\lambda_T(I)}\big),
\end{eqnarray}
and
\begin{eqnarray}\label{H-lemma-2-eq-2}
\sup_{\psi\in{{\mathcal{S}}^M_{\e}}}\int_{\mathbb{Z}\times I}|h(z)\psi(z,s)|\nu(dz)ds\leq\rho_h(\beta)\sqrt{\lambda_T(I)}+\Gamma_h(\beta)a(\e).
\end{eqnarray}
\end{lem}

\begin{lem}\label{H-lemma-3}
Let $h\in L^2(\nu)\cap\mathcal{H}$ be positive. Then for any $\beta>0$,
\begin{eqnarray}\label{H-lemma-3-eq}
\lim_{\e\rightarrow}\sup_{\varphi\in{{\mathcal{S}}^M_{\e}}}\int_{\mathbb{Z}\times I}|h(z)\psi(z,s)|1_{\{|\psi|>\beta/a(\e)\}}\nu(dz)ds=0.
\end{eqnarray}
\end{lem}

\subsection{The proof of MDP-1}
\begin{prp}\label{MDP-1}
If $g^\e\rightarrow g$ in $B_2(R)$, then $\mathcal{G}^0(g^\e)\rightarrow\mathcal{G}^0(g)$ in $C([0,T],\mathbb{V})$.
\end{prp}
\begin{proof}
Set $\mathcal{G}^0(g^\e)=\eta^\e$ and $\mathcal{G}^0(g)=\eta$, First, we will prove that there exist $\e_0, C_R, C_{R,\alpha}$ such that
\begin{eqnarray}\label{MDP-1-estimate-1}
\sup_{t\in[0,T]}\|\eta^\e(t)\|_{\mathbb{W}}^2\leq C_R,
\end{eqnarray}
and for $\alpha\in(0,\frac{1}{2})$
\begin{eqnarray}\label{MDP-1-estimate-2}
\|\eta^\e\|_{W^{\alpha,2}([0,T],{\mathbb{W}}^*)}^2\leq C_{R,\alpha}.
\end{eqnarray}
By (\ref{Theorem-1-eq-1}), we have

\begin{eqnarray}\label{MDP-1-estimate-1-eq-1}
d\big(\eta^\e(t),e_i\big)_{\mathbb{W}}
&=&
-\kappa\big(\widehat{A}\eta^\e(t),e_i\big)_{\mathbb{W}}dt-
%\frac{1}{a(\e)}\Big(\widehat{B}(u^{\e}(t),u^{\e}(t))-\widehat{B}(u^0(t),u^0(t))\Big)dt
\Big(\big(\widehat{B}(\eta^\e(t),u^0(t)),e_i\big)_{\mathbb{W}}+\big(\widehat{B}(u^0(t),\eta^\e(t)),e_i\big)_{\mathbb{W}}\Big)dt\nonumber\\
& &+\big(\widehat{F}'(u^{0}(t),t)\eta^\e(t),e_i\big)_{\mathbb{W}}dt
+\int_{\mathbb{Z}}\big(\widehat{G}(u^{0}(t),z)g^\e(z,t),e_i\big)_{\mathbb{W}}\nu(dz)dt.
\end{eqnarray}
By a simple calculation, we know the fact:
\begin{eqnarray*}
\Big(\widehat{B}(u,v),u\Big)_{\mathbb{W}}=0, \text{ for any }u,v\in\mathbb{W}\cap\mathbb{H}^4(\mathcal{O}).
\end{eqnarray*}
Then, applying the chain rule to $\big(\eta^\e(t),e_i\big)_{\mathbb{W}}^2$ and summing over $i$ from 1 to $\infty$ yields
\begin{eqnarray}\label{MDP-1-estimate-1-eq-2}
& &\|\eta^\e(t)\|^2_{\mathbb{W}}+\frac{2\kappa}{\alpha}\int_0^t\|\eta^\e(s)\|_{\mathbb{W}}^2ds\nonumber\\
&=&
\frac{2\kappa}{\alpha}\int_0^t\Big(curl\big(\eta^\e(t)\big),curl\big(\eta^\e(s)-\alpha\Delta\eta^\e(s)\big)\Big)ds
%\frac{1}{a(\e)}\Big(\widehat{B}(u^{\e}(t),u^{\e}(t))-\widehat{B}(u^0(t),u^0(t))\Big)dt
-2\int_0^t\big(\widehat{B}(u^0(s),\eta^\e(s)),\eta^\e(s)\big)_{\mathbb{W}}ds\nonumber\\
& &+2\int_0^t\big(\widehat{F}'(u^{0}(s),s)\eta^\e(s),\eta^\e(s)\big)_{\mathbb{W}}dt
+2\int_0^t\int_{\mathbb{Z}}\big(\widehat{G}(u^{0}(s),z)g^\e(z,s),\eta^\e(s)\big)_{\mathbb{W}}\nu(dz)ds\nonumber\\
&=&I_1(t)+I_2(t)+I_3(t)+I_4(t).
\end{eqnarray}
Noticing the fact( see (4.61) in \cite{RS-10}):
\begin{eqnarray}
|curl(v)|^2\leq\frac{2}{\alpha}\|v\|_{\mathbb{V}}\ \ \text{ for any }v\in\mathbb{W},
\end{eqnarray}
we have
\begin{eqnarray}\label{MDP-1-estimate-1-eq-3}
I_1(t)
\leq
C\int_0^t\|\eta^\e(s)\|_{\mathbb{W}}^2ds.
\end{eqnarray}
By Condition {\bf (I)}, interpolation inequality and a straightforward calculation, we have
\begin{eqnarray}\label{MDP-1-estimate-1-eq-4}
I_2(t)
&\leq&
C\int_0^t\Big(-\Delta\big(u^0(s)-\alpha\Delta u^0(s)\big)\times\eta^\e(s),curl\big(\eta^\e(s)-\alpha\Delta\eta^\e(s)\big)\Big)\nonumber\\
&\leq&
C\int_0^t\|u^0(s)\|_{\mathbb{H}^4(\mathcal{O})}\|\eta^\e(s)\|_{L^\infty(\mathcal{O})}\|\eta^\e(s)\|_{\mathbb{W}}ds\nonumber\\
&\leq&
C\int_0^t\|\eta^\e(s)\|_{\mathbb{W}}^2ds.
\end{eqnarray}
By Condition {\bf(F2)}, we have
\begin{eqnarray}\label{MDP-1-estimate-1-eq-5}
I_3(t)
\leq
C\int_0^t\|\eta^\e(s)\|_{\mathbb{W}}^2ds.
\end{eqnarray}
By Condition {\bf(G)}, we have
\begin{eqnarray}\label{MDP-1-estimate-1-eq-6}
I_4(t)
&\leq&
C\int_0^t\int_{\mathbb{Z}}M_G(z)\big(1+\|u^0(s)\|_{\mathbb{W}}\big)|g^\e(s,z)|\|\eta^\e(s)\|_{\mathbb{W}}\nu(dz)ds\nonumber\\
&\leq&
C\sup_{s\in[0,T]}\big(1+\|u^0(s)\|_{\mathbb{W}}\big)
\int_0^t\int_{\mathbb{Z}}\big(M^2_G(z)+|g^\e(s,z)|^2\big)\big(1+\|\eta^\e(s)\|_{\mathbb{W}}^2\big)\nu(dz)ds\nonumber\\
&\leq&
C\int_0^T\int_{\mathbb{Z}}M^2_G(z)\nu(dz)ds+C\int_0^T\int_{\mathbb{Z}}|g^\e(s,z)|^2\nu(dz)ds\nonumber\\
& &+C\int_0^t\|\eta^\e(s)\|_{\mathbb{W}}^2ds\int_{\mathbb{Z}}M^2_G(z)\nu(dz)
+C\int_0^t\int_{\mathbb{Z}}|g^\e(s,z)|^2\|\eta^\e(s)\|_{\mathbb{W}}^2\nu(dz)ds\nonumber\\
&\leq&
C(T+R)+C\int_0^t\|\eta^\e(s)\|_{\mathbb{W}}^2\big(1+\int_{\mathbb{Z}}|g^\e(s,z)|^2\nu(dz)\big)ds.
\end{eqnarray}
Combining (\ref{MDP-1-estimate-1-eq-2})-(\ref{MDP-1-estimate-1-eq-6}), we have
\begin{eqnarray}\label{MDP-1-estimate-1-eq-7}
\|\eta^\e(t)\|^2_{\mathbb{W}}+\frac{2\kappa}{\alpha}\int_0^t\|\eta^\e(s)\|_{\mathbb{W}}^2ds
\leq
C(T+R)+C\int_0^t\|\eta^\e(s)\|_{\mathbb{W}}^2\big(1+\int_{\mathbb{Z}}|g^\e(s,z)|^2\nu(dz)\big)ds.
\end{eqnarray}
Applying Gronwall's inequality, we obtain (\ref{MDP-1-estimate-1}).\\
\vskip 0.2cm

\indent Now we prove (\ref{MDP-1-estimate-2}). By (\ref{Theorem-1-eq-1})
\begin{eqnarray}\label{MDP-1-estimate-2-eq-1}
\eta^\e(t)
&=&
-\kappa\int_0^t\widehat{A}\eta^\e(s)ds-
%\frac{1}{a(\e)}\Big(\widehat{B}(u^{\e}(t),u^{\e}(t))-\widehat{B}(u^0(t),u^0(t))\Big)dt
\int_0^t\widehat{B}(\eta^\e(s),u^0(s))ds-\int_0^t\widehat{B}(u^0(s),\eta^\e(s))ds\nonumber\\
& &+\int_0^t\widehat{F}'(u^{0}(s),s)\eta^\e(s)ds+\int_0^t\int_{\mathbb{Z}}\widehat{G}(u^{0}(s),z)g^\e(z,s)\nu(dz)ds\nonumber\\
&=&
I_1+I_2+I_3+I_4+I_5.
\end{eqnarray}
Similarly as the proof of (5.38) in Zhai, Zhang, Zheng \cite{ZZZ1}, we have
\begin{eqnarray}\label{MDP-1-estimate-2-eq-2}
\|I_1+I_2+I_3+I_4\|_{W^{\alpha,2}([0,T],{\mathbb{W}}^*)}^2\leq C_{R,\alpha}.
\end{eqnarray}
For $I_5$, we have
\begin{eqnarray}\label{MDP-1-estimate-2-eq-3}
\|I_5(t)-I_5(s)\|_{\mathbb{V}}^2
&=&
\|\int_s^t\int_{\mathbb{Z}}\widehat{G}(u^{0}(l),z)g^\e(z,l)\nu(dz)dl\|_{\mathbb{V}}^2\nonumber\\
&\leq&
C\Big\{\int_s^t\int_{\mathbb{Z}}M_G(z)\big(1+\|u^{0}(l)\|_{\mathbb{V}}\big)|g^\e(z,l)|\nu(dz)dl\Big\}^2\nonumber\\
&\leq&
C\int_s^t\int_{\mathbb{Z}}M^2_G(z)\big(1+\|u^{0}(l)\|^2_{\mathbb{V}}\big)\int_s^t\int_{\mathbb{Z}}|g^\e(z,l)|^2\nu(dz)dl\nonumber\\
&\leq&
C\sup_{s\in[0,T]}\big(1+\|u^{0}(s)\|^2_{\mathbb{V}}\big)\int_s^t\int_{\mathbb{Z}}M^2_G(z)\nu(dz)dl\int_s^t\int_{\mathbb{Z}}|g^\e(z,l)|^2\nu(dz)dl\nonumber\\
&\leq&
C_R(t-s),
\end{eqnarray}
which implies,
\begin{eqnarray}\label{MDP-1-estimate-2-eq-4}
\|I_5\|_{W^{\alpha,2}([0,T],{\mathbb{W}}^*)}^2\leq C_{R,\alpha}.
\end{eqnarray}
Combining (\ref{MDP-1-estimate-2-eq-2}) and (\ref{MDP-1-estimate-2-eq-4}), we obtain (\ref{MDP-1-estimate-2}).\\

Hence, by (\ref{MDP-1-estimate-1}) and (\ref{MDP-1-estimate-2}), we can assert the existence of element $\widehat{\eta}\in C([0,T],\mathbb{V})\cap L^\infty([0,T],\mathbb{W})$ and a subsequence $\eta^{\e_k}$ such that, as $k\rightarrow\infty$\\
\indent (a) $\sup_{s\in[0,T]}\|\widehat{\eta}(s)\|_\mathbb{W}^{2}\leq C_{R}$,\\
\indent (b) $\eta^{\e_k}\rightarrow \widehat{\eta}$ in $L^2([0,T],\mathbb{W})$ weakly,\\
\indent (c) $\eta^{\e_k}\rightarrow \widehat{\eta}$ in $L^\infty([0,T],\mathbb{V})$ weak-star.\\
Moreover, applying Lemma \ref{Compact}, we have \\
\indent (d) $\eta^{\e_k}\rightarrow \widehat{\eta}$ in $L^2([0,T],\mathbb{V})$ strongly.\\
By the argument as that in the proof of Proposition 4.4 in Zhai, Zhang and Zheng \cite{ZZZ2}, we know $\widehat{\eta}=\eta$.\\
\indent Next, we prove $\eta^{\e_k}\rightarrow \eta$ in $C([0,T],\mathbb{V})$. Let $Z^{\e_k}=\eta^{\e_k}-\eta$, then
\begin{eqnarray}\label{MDP-1-estimate-3-eq-1}
dZ^{\e_k}(t)
&=&
-\kappa\widehat{A}Z^{\e_k}(t)dt-\Big(\widehat{B}(Z^{\e_k}(t),u^0(t))+\widehat{B}(u^0(t),Z^{\e_k}(t))\Big)dt\nonumber\\
& &+\widehat{F}'(u^{0}(t),t)Z^{\e_k}(t)dt+\int_{\mathbb{Z}}\widehat{G}(u^{0}(t),z)\big(g^\e(z,t)-g(z,t)\big)\nu(dz)dt,
\end{eqnarray}
with initial value $Z^{\e_k}(0)=0$.\\
\indent Applying the chain rule, by Lemma \ref{Lem-B-01} and Conditions {\bf (F1), (F2)} and {\bf (G)}, we have
\begin{eqnarray}\label{MDP-1-estimate-3-eq-2}
& &\|Z^{\e_k}(t)\|_{\mathbb{V}}^2+2\kappa\int_0^t\|Z^{\e_k}(s)\|^2ds\nonumber\\
&=&
2\int_0^t\Big(\widehat{B}(Z^{\e_k}(s),Z^{\e_k}(s)),u^0(s)\Big)_{{\mathbb{W}}^*,\mathbb{W}}ds
+2\int_0^t\big(\widehat{F}'(u^{0}(s),s)Z^{\e_k}(s),Z^{\e_k}(s)\big)_{\mathbb{V}}ds\nonumber\\
& &
+2\int_0^t\int_{\mathbb{Z}}\big(\widehat{G}(u^{0}(s),z)\big(g^\e(z,s)-g(z,s)\big),Z^{\e_k}(s)\big)_{\mathbb{V}}\nu(dz)ds\nonumber\\
&\leq&
C\sup_{s\in[0,T]}\|u^0(s)\|_{\mathbb{W}}\int_0^t\|Z^{\e_k}(s)\|_{\mathbb{V}}^2ds
+C\int_0^t\|Z^{\e_k}(s)\|_{\mathbb{V}}^2ds\nonumber\\
& &
+C\sup_{s\in[0,T]}\big(1+\|u^0(s)\|_{\mathbb{V}}\big)
\int_0^t\int_{\mathbb{Z}}M_G(z)\big|g^\e(z,s)-g(z,s)\big|\|Z^{\e_k}(s)\|_{\mathbb{V}}\nu(dz)ds\nonumber\\
&\leq&
C\int_0^t\|Z^{\e_k}(s)\|_{\mathbb{V}}^2ds
+C\big\{\int_0^t\int_{\mathbb{Z}}M^2_G(z)\nu(dz)\|Z^{\e_k}(s)\|_{\mathbb{V}}^2ds\big\}^\frac{1}{2}
\big\{\int_0^t\int_{\mathbb{Z}}\big(g^\e(z,s)-g(z,s)\big)^2\nu(dz)ds\big\}^\frac{1}{2}\nonumber\\
&\leq&
C\int_0^t\|Z^{\e_k}(s)\|_{\mathbb{V}}^2ds
+C_{R}\big\{\int_0^t\|Z^{\e_k}(s)\|_{\mathbb{V}}^2ds\big\}^\frac{1}{2},
\end{eqnarray}
in the last inequality, we have used $M_{G}\in L^2(\nu_T)$ and $g^\e,g\in B_2(R)$.\\
\indent Using (d), it follows that
\begin{eqnarray}\label{MDP-1-estimate-3}
\lim_{\e\rightarrow0}\sup_{t\in[0,T]}\|Z^{\e_k}(t)\|_{\mathbb{V}}\leq\lim_{\e\rightarrow0}\Big\{C\int_0^T\|Z^{\e_k}(s)\|_{\mathbb{V}}^2ds
+C_{R}\big\{\int_0^T\|Z^{\e_k}(s)\|_{\mathbb{V}}^2ds\big\}^\frac{1}{2}\Big\}=0.
\end{eqnarray}
The proof is complete.
\end{proof}

\subsection{The proof of MDP-2}

\indent By Girsonav transform theorem, we can see that the following equation has a unique solution:
\begin{eqnarray}\label{MDP-1-estimate-eq-1}
dX^\e(t)
&=&
-\kappa\widehat{A}X^{\e}(t)dt-\widehat{B}(X^{\e}(t),X^{\e}(t))dt+\widehat{F}(X^{\e}(t),t)\nonumber\\
& &+\e\int_{\mathbb{Z}}\widehat{G}(X^{\e}(t-),z)\widetilde{N}^{\e^{-1}\varphi^\e}(dzdt)
+\int_{\mathbb{Z}}\widehat{G}(X^{\e}(t),z)\big(\varphi^\e(t,z)-1\big)\nu(dz)dt.
\end{eqnarray}
with initial value $X^\e(0)=u_0$. We need the following estimates:
\begin{lem}\label{MDP-2-prp-1}
There exists a $\e_0>0$ such that
\begin{eqnarray}\label{MDP-2-estimate-1}
\sup_{\e\in(0,\e_0)}E\Big[\|X^{\e}(t)\|_\mathbb{W}^2\Big]\leq C_{\e_0}<\infty.
\end{eqnarray}
\end{lem}

\indent Let $\Pi_n$ be the projection operator from $\mathbb{W}$ to $\mathbb{W}$ defined as
\begin{eqnarray*}
\Pi_{n}u=\sum_{i=1}^n\big(u,e_i\big)_{\mathbb{W}}e_i,\ \ u\in\mathbb{W}.\\
\end{eqnarray*}
\indent Set $\mathbb{W}_n=Span\{e_1,\cdots,e_n\}$. Let $\widetilde{X}^{n,\e}\in{\mathbb{W}}_n$ be the Garlerkin approximations of (\ref{MDP-1-estimate-eq-1}) satisfying

\begin{eqnarray}\label{MDP-1-estimate-eq-2}
& &d\big(X^{n,\e}(s),e_i\big)_\mathbb{V}
+\kappa\big(\Pi_{n}\widehat{A}X^{n,\e}(s),e_i\big)_{\mathbb{V}}ds\nonumber\\
&=&
-\big(\Pi_{n}\widehat{B}(X^{n,\e}(s),X^{n,\e}(s)),e_i\big)_{\mathbb{V}}ds
+\big(\Pi_{n}\widehat{F}(X^{n,\e}(s),s),e_i\big)_{\mathbb{V}}ds\nonumber\\
& &
+\big(\int_{\mathbb{Z}}\Pi_{n}\widehat{G}(X^{n,\e}(s),z)(\varphi_\e(s,z)-1)\nu(dz),e_i\big)_{\mathbb{V}}ds\nonumber\\
& &
+\e\int_{\mathbb{Z}}\big(\Pi_{n}\widehat{G}(X^{n,\e}(s-),z),e_i\big)_{\mathbb{V}}\widetilde{N}^{\e^{-1}\varphi_\e}(dzds),
\end{eqnarray}
for any $i=1,2,\cdots,n.$\\
\indent As in the proof of Theorem 3.2 in Shang, Zhai and Zhang \cite{SZZ}, one can show that $\lim_{n\rightarrow\infty}X^{n,\e}=X^{\e}$ with respect to the weak topology in $L^4(\Omega,\mathcal{F},P;L^\infty([0,T];\mathbb{W}))$. Hence Lemma \ref{MDP-2-prp-1} will follow from the following result.

\begin{lem}\label{MDP-2-prp-2}
There exists $\e_0>0$ such that
\begin{eqnarray}\label{MDP-2-estimate-2}
\sup_{\e\in(0,\e_0)}E\Big[\sup_{t\in[0,T]}\|X^{n,\e}(t)\|_\mathbb{W}^{2}\Big]\leq C_{\e_0}<\infty.
\end{eqnarray}

\end{lem}

\begin{proof}
Multiplying $\lambda_i$ at both sides of the equation (\ref{MDP-1-estimate-eq-2}), we can use (\ref{Basis}) to obtain
\begin{eqnarray}\label{MDP-1-estimate-eq-3}
& &d\big(X^{n,\e}(s),e_i\big)_\mathbb{W}
+\kappa\big(\widehat{A}X^{n,\e}(s),e_i\big)_{\mathbb{W}}ds\nonumber\\
&=&
-\big(\widehat{B}(X^{n,\e}(s),X^{n,\e}(s)),e_i\big)_{\mathbb{W}}ds
+\big(\widehat{F}(X^{n,\e}(s),s),e_i\big)_{\mathbb{W}}ds\nonumber\\
& &
+\big(\int_{\mathbb{Z}}\widehat{\sigma}(X^{n,\e}(s),z)(\varphi_\e(s,z)-1)\nu(dz),e_i\big)_{\mathbb{W}}ds\nonumber\\
& &
+\e\int_{\mathbb{Z}}\big(\widehat{\sigma}(X^{n,\e}(s-),z),e_i\big)_{\mathbb{W}}\widetilde{N}^{\e^{-1}\varphi_\e}(dzds),
\end{eqnarray}
for any $i\in\mathbb{N}$.\\
\indent Applying $It\hat{o}$'s formula to $\big(X^{n,\e}(s),e_i\big)_\mathbb{W}^2$ and then summing over $i$ from $1$ to $n$ yields
\begin{eqnarray}\label{MDP-1-estimate-eq-4}
\|X^{n,\e}(t)\|_\mathbb{W}^2
&=&
\|X_0\|_\mathbb{W}^2-2\kappa\int_0^t\big(\widehat{A}X^{n,\e}(s),X^{n,\e}(s)\big)_{\mathbb{W}}ds\nonumber\\
& &
-2\int_0^t\big(\widehat{B}(X^{n,\e}(s),X^{n,\e}(s)),X^{n,\e}(s)\big)_{\mathbb{W}}ds
+2\int_0^t\big(\widehat{F}(X^{n,\e}(s),s),X^{n,\e}(s)\big)_{\mathbb{W}}ds\nonumber\\
& &
+2\int_0^t\big(\int_{\mathbb{Z}}
\widehat{G}(X^{n,\e}(s),z)(\varphi_\e(s,z)-1)\nu(dz),X^{n,\e}(s)\big)_{\mathbb{W}}ds\nonumber\\
& &
+2\e\int_0^t\int_{\mathbb{Z}}\big(
\widehat{G}(X^{n,\e}(s-),z),X^{n,\e}(s-)\big)_{\mathbb{W}}\widetilde{N}^{\e^{-1}\varphi_\e}(dzds)\nonumber\\
& &
+\e^2\int_0^t\int_{\mathbb{Z}}\|\widehat{G}(X^{n,\e}(s-),z)\|_{\mathbb{W}}^2N^{\e^{-1}\varphi_\e}(dzds).
\end{eqnarray}
By a simple calculation, we get
\begin{eqnarray}\label{MDP-1-estimate-eq-5}
& &\|X^{n,\e}(t)\|_\mathbb{W}^2
+\frac{2\kappa}{\alpha}\int_0^t\|X^{n,\e}(s)\|_\mathbb{W}^2ds\nonumber\\
&=&
\|X_0\|_\mathbb{W}^2+\frac{2\kappa}{\alpha}\int_0^t
\Big(curl\big(X^{n,\e}(s)\big),curl\big(X^{n,\e}(s)-\alpha\Delta X^{n,\e}(s)\big)\Big)ds\nonumber\\
& &
+2\int_0^t\big(\widehat{F}(X^{n,\e}(s),s),X^{n,\e}(s)\big)_{\mathbb{W}}ds
+2\int_0^t\big(\int_{\mathbb{Z}}
\widehat{G}(X^{n,\e}(s),z)(\varphi_\e(s,z)-1)\nu(dz),X^{n,\e}(s)\big)_{\mathbb{W}}ds\nonumber\\
& &
+2\e\int_0^t\int_{\mathbb{Z}}\big(
\widehat{G}(X^{n,\e}(s-),z),X^{n,\e}(s-)\big)_{\mathbb{W}}\widetilde{N}^{\e^{-1}\varphi_\e}(dzds)\nonumber\\
& &
+\e^2\int_0^t\int_{\mathbb{Z}}\|\widehat{G}(X^{n,\e}(s-),z)\|_{\mathbb{W}}^2N^{\e^{-1}\varphi_\e}(dzds)¡£
\end{eqnarray}
We have
\begin{eqnarray}\label{MDP-1-estimate-eq-6}
& &\frac{2\kappa}{\alpha}\int_0^t
\Big(curl\big(X^{n,\e}(s)\big),curl\big(X^{n,\e}(s)-\alpha\Delta X^{n,\e}(s)\big)\Big)ds+2\int_0^t\big(\widehat{F}(X^{n,\e}(s),s),X^{n,\e}(s)\big)_{\mathbb{W}}ds\nonumber\\
&\leq&
C \int_0^{t}\|X^{n,\e}(s)\|_\mathbb{W}\|X^{n,\e}(s)\|_\mathbb{V}ds
+C \int_0^{t}\|X^{n,\e}(s)\|_\mathbb{W}\|\widehat{F}(X^{n,\e}(s),s)\|_\mathbb{W}ds\nonumber\\
&\leq&
C\int_0^t \|X^{n,\e}(s)\|_\mathbb{W}^2ds.
\end{eqnarray}
Set $\psi_\e(s,z)=\big(\varphi_\e(s,z)-1\big)/a(\e)\in{\mathcal{U}}_{\e}^M$. Then
\begin{eqnarray}\label{MDP-1-estimate-eq-7}
& &\Big|2\int_0^{t}\big(\int_{\mathbb{Z}}
\widehat{G}(X^{n,\e}(s),z)(\varphi_\e(s,z)-1)\nu(dz),X^{n,\e}(s)\big)_{\mathbb{W}}ds\Big|\nonumber\\
&\leq&
2a(\e)\int_0^t\|X^{n,\e}(s)\|_{\mathbb{W}}\int_{\mathbb{Z}}\|\widehat{G}(X^{n,\e}(s),z)\|_{\mathbb{W}}|\psi_\e(s,z)|\nu(dz)ds\nonumber\\
&\leq&
2a(\e)\int_0^t\|X^{n,\e}(s)\|_{\mathbb{W}}\big(1+\|X^{n,\e}(s)\|_{\mathbb{W}}\big)\int_{\mathbb{Z}}M_G(z)|\psi_\e(s,z)|\nu(dz)ds\nonumber\\
&\leq&
4a(\e)\int_0^t\big(1+\|X^{n,\e}(s)\|_{\mathbb{W}}^2\big)\int_{\mathbb{Z}}M_G(z)|\psi_\e(s,z)|\nu(dz)ds.
\end{eqnarray}
Combining (\ref{MDP-1-estimate-eq-5})-(\ref{MDP-1-estimate-eq-7}), we have
\begin{eqnarray}\label{MDP-1-estimate-eq-8}
& &\|X^{n,\e}(t)\|_\mathbb{W}^2
+\frac{2\kappa}{\alpha}\int_0^t\|X^{n,\e}(s)\|_\mathbb{W}^2ds\nonumber\\
&\leq&
\|X_0\|_\mathbb{W}^2+\sup_{l\in[0,T]}\Big|2\e\int_0^l\int_{\mathbb{Z}}\big(
    \widehat{G}(X^{n,\e}(s-),z),X^{n,\e}(s-)\big)_{\mathbb{W}}\widetilde{N}^{\e^{-1}\varphi_\e}(dzds)\Big|\nonumber\\
& &
+\e^2\int_0^T\int_{\mathbb{Z}}\|\widehat{G}(X^{n,\e}(s-),z)\|_{\mathbb{W}}^2N^{\e^{-1}\varphi_\e}(dzds)
+4a(\e)\int_0^T\int_{\mathbb{Z}}M_G(z)|\psi_\e(s,z)|\nu(dz)ds\nonumber\\
& &
+\int_0^t \|X^{n,\e}(s)\|_\mathbb{W}^2\big(C+4a(\e)\int_{\mathbb{Z}}M_G(z)|\psi_\e(s,z)|\nu(dz)\big)ds\nonumber\\
&=&
I_1+I_2+I_3+I_4+I_5(t).
\end{eqnarray}
Applying Gronwall's inequality and using Lemma \ref{H-lemma-2}, we get
\begin{eqnarray}\label{MDP-1-estimate-eq-9}
& &\|X^{n,\e}(t)\|_\mathbb{W}^2
+\frac{2\kappa}{\alpha}\int_0^t\|X^{n,\e}(s)\|_\mathbb{W}^2ds\nonumber\\
&\leq&
\big(I_1+I_2+I_3+I_4\big)\exp\big\{CT+4a(\e)\big(\rho_{M_G}(\beta)\sqrt{T}+\Gamma_{M_G}(\beta)a(\e)\big)\big\}.
\end{eqnarray}
By Lemma \ref{H-lemma-2} again, we get
\begin{eqnarray}\label{MDP-1-estimate-eq-10}
I_1+I_4
\leq
C+4a(\e)\big(\rho_{M_G}(\beta)\sqrt{T}+\Gamma_{M_G}(\beta)a(\e)\big).
\end{eqnarray}
\noindent By B-D-G and Young's inequalities, (\ref{G-02}) and Lemma \ref{H-lemma-1}, we get
\begin{eqnarray}\label{MDP-1-estimate-eq-11}
EI_2
&\leq&
2\e E\Big[\Big(\int_0^{T}
\int_{\mathbb{Z}}\big(\widehat{G}(s,X^{n,\e}(s-),z),
X^{n,\e}(s-)\big)_{\mathbb{W}}^2N^{\e^{-1}\varphi_\e}(dzds)\Big)^\frac{1}{2}\Big]\nonumber\\
&\leq&
2\e E\Big[\Big(\sup_{s\in[0,T]}\|X^{n,\e}(s)\|_\mathbb{W}^2\Big)^\frac{1}{2}
\Big(\int_0^T\int_{\mathbb{Z}}\|\widehat{G}(s,X^{n,\e}(s-),z)\|_{\mathbb{W}}^2N^{\e^{-1}\varphi_\e}(dzds)\Big)^\frac{1}{2}\Big]\nonumber\\
&\leq&
\frac{1}{4}E\Big[\sup_{t\in[0,T]}\|X^{n,\e}(t)\|_\mathbb{W}^2\Big]
+16{\e}^{2}E\Big[\Big(\int_0^T\int_{\mathbb{Z}}\|\widehat{G}(s,X^{n,\e}(s-),z)\|_{\mathbb{W}}^2N^{\e^{-1}\varphi_\e}(dzds)\Big)\Big]\nonumber\\
&\leq&
\frac{1}{4}E\Big[\sup_{t\in[0,T]}\|X^{n,\e}(t)\|_\mathbb{W}^2\Big]
+C\e\varsigma_{M_G}\big(a^2(\e)+T\big)\Big(E\Big[\sup_{t\in[0,T]}\|X^{n,\e}(t)\|_\mathbb{W}^2\Big]+1\Big).
\end{eqnarray}
Similar to (\ref{MDP-1-estimate-eq-11}), we get
\begin{eqnarray}\label{MDP-1-estimate-eq-12}
EI_4
&=&
\e\int_0^T\int_{\mathbb{Z}}\|\widehat{G}(X^{n,\e}(s-),z)\|_{\mathbb{W}}^2\varphi_\e(s,z)\nu(dz)ds\nonumber\\
&\leq&
C\e\varsigma_{M_G}\big(a^2(\e)+T\big)\Big(E\Big[\sup_{t\in[0,T]}\|X^{n,\e}(t)\|_\mathbb{W}^2\Big]+1\Big).
\end{eqnarray}
Choosing $\e_0>0$ small enough such that $C\e_0\varsigma_{M_G}\big(a^2(\e_0)+T\big)\leq\frac{1}{8}$, and combining (\ref{MDP-1-estimate-eq-9})-(\ref{MDP-1-estimate-eq-12}), we obtain (\ref{MDP-2-estimate-2}). The proof is complete.

\end{proof}

Recall $u_0$ in (\ref{Deterministic}). We have
\begin{thm}\label{MDP-2-theorem-1}
\begin{eqnarray}\label{MDP-2-theorem-1-estimate}
\lim_{\e\rightarrow0}E\Big[\sup_{t\in[0,T]}\|X^{\e}(t)-u^0(t)\|_\mathbb{V}^2\Big]=0.
\end{eqnarray}
\end{thm}
\begin{proof}
Set $Z^\e(t)=X^\e(t)-u^0(t)$. Then
\begin{eqnarray}\label{MDP-2-theorem-1-eq-1}
dZ^\e(t)
&=&
-\kappa\widehat{A}Z^{\e}(t)dt-\widehat{B}(X^{\e}(t),Z^{\e}(t))dt-\widehat{B}(Z^{\e}(t),u^{0}(t))dt
+\big(\widehat{F}(X^{\e}(t),t)-\widehat{F}(u^{0}(t),t)\big)dt\nonumber\\
& &+\e\int_{\mathbb{Z}}\widehat{G}(X^{\e}(t-),z)\widetilde{N}^{\e^{-1}\varphi^\e}(dzdt)
+\int_{\mathbb{Z}}\widehat{G}(X^{\e}(t),z)\big(\varphi^\e(t,z)-1\big)\nu(dz)dt,
\end{eqnarray}
with initial value $Z^\e(0)=0$.
Applying It$\hat{o}$'s formula and (\ref{Eq B-03}), we get
\begin{eqnarray}\label{MDP-2-theorem-1-eq-2}
& &d\|Z^\e(t)\|_{\mathbb{V}}^2+2\kappa\|Z^\e(t)\|^2dt\nonumber\\
&=&
-2\langle\widehat{B}(Z^{\e}(t),u^{0}(t)),Z^\e(t)\rangle_{\mathbb{W}^*,\mathbb{W}}dt
+2\big(\widehat{F}(X^{\e}(t),t)-\widehat{F}(u^{0}(t),t),Z^\e(t)\big)_{\mathbb{V}}dt\nonumber\\
& &
+2\e\int_{\mathbb{Z}}\big(\widehat{G}(X^{\e}(t-),z),Z^\e(t-)\big)_{\mathbb{V}}\widetilde{N}^{\e^{-1}\varphi^\e}(dzdt)
+2\int_{\mathbb{Z}}\big(\widehat{G}(X^{\e}(t),z)\big(\varphi^\e(t,z)-1\big),Z^\e(t)\big)_{\mathbb{V}}\nu(dz)dt\nonumber\\
& &
+\e^2\int_{\mathbb{Z}}\|\widehat{G}(X^{\e}(t-),z)\|_{\mathbb{V}}^2{N}^{\e^{-1}\varphi^\e}(dzdt).
\end{eqnarray}
By Lemma \ref{Lem-B-01} and Condition {\bf (F1)}, we get
\begin{eqnarray}\label{MDP-2-theorem-1-eq-3}
\int_0^t2|\langle\widehat{B}(Z^{\e}(s),u^{0}(s)),Z^\e(s)\rangle_{\mathbb{W}^*,\mathbb{W}}|ds,
\leq
C\sup_{l\in[0,T]}\|u^0(l)\|_{\mathbb{W}}\int_0^t\|Z^\e(s)\|_{\mathbb{V}}^2ds,
\end{eqnarray}
and
\begin{eqnarray}\label{MDP-2-theorem-1-eq-4}
2\int_0^t|\big(\widehat{F}(X^{\e}(s),s)-\widehat{F}(u^{0}(s),s),Z^\e(s)\big)_{\mathbb{V}}|ds
\leq
C\int_0^t\|Z^\e(s)\|_{\mathbb{V}}^2ds.
\end{eqnarray}
Set $\psi_\e(s,z)=\big(\varphi_\e(s,z)-1\big)/a(\e)$. By Condition {\bf (G)},
\begin{eqnarray}\label{MDP-2-theorem-1-eq-5}
& &2\int_0^t|\int_{\mathbb{Z}}\big(\widehat{G}(X^{\e}(s),z)\big(\varphi^\e(s,z)-1\big),Z^\e(s)\big)_{\mathbb{V}}\nu(dz)ds\nonumber\\
&\leq&
2\int_0^t\|Z^\e(s)\|_{\mathbb{V}}\int_{\mathbb{Z}}
    \|\widehat{G}(X^{\e}(s),z)-\widehat{G}(u^0(s),z)\|_{\mathbb{V}}|\varphi^\e(s,z)-1|\nu(dz)ds\nonumber\\
& &+2\int_0^t\|Z^\e(s)\|_{\mathbb{V}}\int_{\mathbb{Z}}\|\widehat{G}(u^0(s),z)\|_{\mathbb{V}}|\varphi^\e(t,z)-1|\nu(dz)ds\nonumber\\
&\leq&
Ca(\e)\int_0^t\|Z^\e(s)\|_{\mathbb{V}}^2\int_{\mathbb{Z}}L_G(z)|\psi_\e(s,z)|\nu(dz)ds\nonumber\\
& &+Ca(\e)\int_0^t\big(1+\|Z^\e(s)\|_{\mathbb{V}}^2\big)\big(1+\|u^0(s)\|_{\mathbb{V}}\big)
\int_{\mathbb{Z}}M_G(z)|\psi_\e(s,z)|\nu(dz)ds\nonumber\\
&\leq&
Ca(\e)\int_0^t\|Z^\e(s)\|_{\mathbb{V}}^2
    \int_{\mathbb{Z}}\big(L_G(z)+(1+\sup_{l\in[0,T]}\|u^0(l)\|_{\mathbb{V}})M_G(z)\big)|\psi_\e(s,z)|\nu(dz)ds\nonumber\\
& &+Ca(\e)\big(1+\sup_{l\in[0,T]}\|u^0(l)\|_{\mathbb{V}}\big)\int_0^t\int_{\mathbb{Z}}M_G(z)|\psi_\e(s,z)|\nu(dz)ds.
\end{eqnarray}
Combining (\ref{MDP-2-theorem-1-eq-2})-(\ref{MDP-2-theorem-1-eq-5}), we get
\begin{eqnarray}\label{MDP-2-theorem-1-eq-6}
\|Z^\e(t)\|_{\mathbb{V}}^2+2\kappa\int_0^t\|Z^\e(s)\|^2ds
\leq
M_1(T)+M_2(T)+M_3(T)+\int_0^tJ(s)\|Z^\e(s)\|_{\mathbb{V}}^2ds,
\end{eqnarray}
here
\begin{eqnarray*}
M_1(T)=2\e\sup_{l\in[0,T]}\Big|\int_0^l
\int_{\mathbb{Z}}\big(\widehat{G}(X^{\e}(s-),z),Z^\e(s-)\big)_{\mathbb{V}}\widetilde{N}^{\e^{-1}\varphi^\e}(dzds)\Big|,
\end{eqnarray*}
\begin{eqnarray*}
M_2(T)=\e^2\int_0^T\int_{\mathbb{Z}}\|\widehat{G}(X^{\e}(t-),z)\|_{\mathbb{V}}^2{N}^{\e^{-1}\varphi^\e}(dzdt),
\end{eqnarray*}
\begin{eqnarray*}
M_3(T)=Ca(\e)\big(1+\sup_{l\in[0,T]}\|u^0(l)\|_{\mathbb{V}}\big)\int_0^T\int_{\mathbb{Z}}M_G(z)|\psi_\e(s,z)|\nu(dz)ds,
\end{eqnarray*}
and
\begin{eqnarray*}
J(s)
&=&
C\Big(\sup_{l\in[0,T]}\|u^0(l)\|_{\mathbb{W}}+1+a(\e)\int_{\mathbb{Z}}L_G(z)|\psi_\e(s,z)|\nu(dz)\nonumber\\
& &
+a(\e)(1+\sup_{l\in[0,T]}\|u^0(l)\|_{\mathbb{V}})\int_{\mathbb{Z}}M_G(z)|\psi_\e(s,z)|\nu(dz)\Big).
\end{eqnarray*}
By Gronwall's inequality and Lemma \ref{H-lemma-2} and Lemma \ref{Regularity},
\begin{eqnarray}\label{MDP-2-theorem-1-eq-7}
& &\|Z^\e(t)\|_{\mathbb{V}}^2+2\kappa\int_0^t\|Z^\e(s)\|^2ds\nonumber\\
&\leq&
\Big(M_1(T)+M_2(T)+M_3(T)\Big)\exp{\big(\int_0^TJ(s)ds\big)}\nonumber\\
&\leq&
C\Big(M_1(T)+M_2(T)+M_3(T)\Big).
\end{eqnarray}
By B-D-G inequality, Lemma \ref{H-lemma-1} and (\ref{MDP-2-estimate-1}), we get
\begin{eqnarray}\label{MDP-2-theorem-1-eq-8}
EM_1(T)
&\leq&
2\e E\Big[\Big(\int_0^{T}
\int_{\mathbb{Z}}\big(\widehat{G}(X^{\e}(s-),z),
Z^{\e}(s-)\big)_{\mathbb{V}}^2N^{\e^{-1}\varphi_\e}(dzds)\Big)^\frac{1}{2}\Big]\nonumber\\
&\leq&
2\e E\Big[\Big(\sup_{s\in[0,T]}\|Z^{\e}(s)\|_\mathbb{V}^2\Big)^\frac{1}{2}
\Big(\int_0^T\int_{\mathbb{Z}}\|\widehat{G}(X^{\e}(s-),z)\|_{\mathbb{V}}^2N^{\e^{-1}\varphi_\e}(dzds)\Big)^\frac{1}{2}\Big]\nonumber\\
&\leq&
\frac{1}{2}E\Big[\sup_{t\in[0,T]}\|Z^{\e}(t)\|_\mathbb{V}^2\Big]
+C{\e}E\Big[\Big(\int_0^T\int_{\mathbb{Z}}M^2_G(z)\big(1+\|X^\e(s)\|_{\mathbb{V}}^2\big)\varphi_\e(s,z)\nu(dz)ds\Big)\Big]\nonumber\\
&\leq&
\frac{1}{2}E\Big[\sup_{t\in[0,T]}\|Z^{\e}(t)\|_\mathbb{V}^2\Big]
+C\e\varsigma_{M_G}\big(a^2(\e)+T\big)\Big(E\Big[\sup_{t\in[0,T]}\|X^{\e}(t)\|_\mathbb{V}^2\Big]+1\Big)\nonumber\\
&\leq&
\frac{1}{2}E\Big[\sup_{t\in[0,T]}\|Z^{\e}(t)\|_\mathbb{V}^2\Big]+C\e\varsigma_{M_G}\big(a^2(\e)+T\big).
\end{eqnarray}
Similarly, we have
\begin{eqnarray}\label{MDP-2-theorem-1-eq-9}
EM_2(T)
&=&
\e E\Big[\int_0^T\int_{\mathbb{Z}}\|\widehat{G}(X^{\e}(t-),z)\|_{\mathbb{V}}^2\varphi^\e(t,z)\nu(dz)dt\nonumber\\
&\leq&
C\e E\Big[\Big(\int_0^T\int_{\mathbb{Z}}M^2_G(z)\big(1+\|X^\e(s)\|_{\mathbb{V}}^2\big)\varphi_\e(s,z)\nu(dz)ds\Big)\Big]\nonumber\\
&\leq&
C\e\varsigma_{M_G}\big(a^2(\e)+T\big).
\end{eqnarray}
By Lemma \ref{H-lemma-2} and Lemma \ref{Regularity},
\begin{eqnarray}\label{MDP-2-theorem-1-eq-10}
EM_3(T)
\leq
Ca(\e)\big(\rho_{M_G}(\beta)\sqrt{T}+\Gamma_{M_G}(\beta)a(\e)\big).
\end{eqnarray}
Combining (\ref{MDP-2-theorem-1-eq-7})-(\ref{MDP-2-theorem-1-eq-10}), we obtain (\ref{MDP-2-theorem-1-estimate}).\\
\indent The proof is complete.

\end{proof}

Notice
\begin{eqnarray}
\mathcal{G}^{\e}(\e N^{\e^{-1}\varphi^\e}):=Y^\e=\frac{X^\e-u^0}{a(\e)}.
\end{eqnarray}
and $Y^\e(t)$ satisfies
\begin{eqnarray}\label{4.1}
dY^\e(t)
&=&
-\kappa\widehat{A}Y^\e(t)dt-
\Big(\widehat{B}(X^{\e}(t),Y^\e(t))+\widehat{B}(Y^\e(t),u^0(t))\Big)dt
+\frac{1}{a(\e)}\Big(\widehat{F}(X^{\e}(t),t)-\widehat{F}(u^0(t),t)\Big)dt\nonumber\\
& &\qquad+\frac{\e}{a(\e)}\int_{\mathbb{Z}}\widehat{G}(X^{\e}(t-),z)\widetilde{N}^{\e^{-1}\varphi^\e}(dzdt)
+\int_{\mathbb{Z}}\widehat{G}(X^{\e}(t),z)\big(\frac{\varphi^\e(t,z)-1}{a(\e)}\big)\nu(dz)dt,
\end{eqnarray}
with initial value $Y^\e(t)=0$.

\begin{prp}\label{MDP-2}
Given $M<\infty$, Let $\{\varphi^\e\}_{\e>0}$ be such that $\varphi^\e\in\mathcal{U}_{+,\e}^M$ for every $\e>0$. Let $\psi^\e=(\varphi^\e-1)/a(\e)$ and $\beta\in(0,1]$. Then the family $\{Y^\e,\psi^\e 1_{\{|\psi^\e|\leq\beta/a(\e)\}}\}_{\e>0}$ is tight in $D([0,T],\mathbb{V})\times B_2\big(\sqrt{M\kappa_2(1)}\big)$, and any limit point $(Y,\psi)$ solves the equation (\ref{Theorem-1-eq-1}).
\end{prp}
\begin{proof}
The proof is divided into several steps.\\
\indent \textbf{Step 1.} Let $Z^\e$ be the solution of the following equation
\begin{eqnarray}\label{MDP-2-eq-1-1}
dZ^\e(t)=-\kappa\widehat{A}Z^\e(t)dt+\frac{\e}{a(\e)}\int_{\mathbb{Z}}\widehat{G}(X^\e(t-),z)\widetilde{N}^{\e^{-1}\varphi^\e}(dzdt),
\end{eqnarray}
with initial value $Z^\e(0)=0$.
Then, by (\ref{Basis})
\begin{eqnarray}\label{MDP-2-eq-1-2}
d\big(Z^\e(t),e_i\big)_{\mathbb{W}}=-\kappa\big(\widehat{A}Z^\e(t),e_i\big)_{\mathbb{W}}dt
+\frac{\e}{a(\e)}\int_{\mathbb{Z}}\big(\widehat{G}(X^\e(t-),z),e_i\big)_{\mathbb{W}}\widetilde{N}^{\e^{-1}\varphi^\e}(dzdt),
\end{eqnarray}
for $i\in\mathbb{N}$.\\
Applying It$\hat{o}$'s formula to $\big(Z^\e(t),e_i\big)_{\mathbb{W}}^2$ and summing over $i$ from 1 to $\infty$ yields
\begin{eqnarray}\label{MDP-2-eq-1-3}
& &\|Z^\e(t)\|_{\mathbb{W}}^2+\frac{2\kappa}{\alpha}\int_0^t\|Z^\e(s)\|_{\mathbb{W}}^2ds\nonumber\\
&=&
\frac{2\kappa}{\alpha}\int_0^t\Big(curl\big(Z^\e(s)\big),curl\big(Z^\e(s)-\alpha\Delta Z^\e(s)\big)\Big)ds
+\frac{2\e}{a(\e)}\int_0^t\int_{\mathbb{Z}}\big(\widehat{G}(X^\e(s-),z),Z^\e(s-)\big)_{\mathbb{W}}\widetilde{N}^{\e^{-1}\varphi^\e}(dzds)\nonumber\\
& &
+\frac{\e^2}{a^2(\e)}\int_0^t\int_{\mathbb{Z}}\|\widehat{G}(X^\e(s-),z)\|_{\mathbb{W}}N^{\e^{-1}\varphi^\e}(dzds).
\end{eqnarray}
We have
\begin{eqnarray}\label{MDP-2-eq-1-4}
\frac{2\kappa}{\alpha}\int_0^t
\Big(curl\big(Z^\e(s)\big),curl\big(Z^\e(s)-\alpha\Delta Z^\e(s)\big)\Big)ds
\leq
C\int_0^t \|Z^\e(s)\|_\mathbb{W}^2ds.
\end{eqnarray}
By B-D-G inequality, Lemma \ref{H-lemma-1} and (\ref{MDP-2-estimate-1}),
\begin{eqnarray}\label{MDP-2-eq-1-5}
& &E\Big[\sup_{t\in[0,T]}\Big|\frac{2\e}{a(\e)}\int_0^t\int_{\mathbb{Z}}
\big(\widehat{G}(X^\e(s-),z),Z^\e(s-)\big)_{\mathbb{W}}\widetilde{N}^{\e^{-1}\varphi^\e}(dzds)\Big|\Big]\nonumber\\
&\leq&
\frac{2\e}{a(\e)}E\Big[\Big(\int_0^{T}
\int_{\mathbb{Z}}\big(\widehat{G}(s,X^{\e}(s-),z),
Z^{\e}(s-)\big)_{\mathbb{W}}^2N^{\e^{-1}\varphi_\e}(dzds)\Big)^\frac{1}{2}\Big]\nonumber\\
&\leq&
\frac{2\e}{a(\e)}E\Big[\Big(\sup_{s\in[0,T]}\|Z^{\e}(s)\|_\mathbb{W}^2\Big)^\frac{1}{2}
\Big(\int_0^T\int_{\mathbb{Z}}\|\widehat{G}(X^{\e}(s-),z)\|_{\mathbb{W}}^2N^{\e^{-1}\varphi_\e}(dzds)\Big)^\frac{1}{2}\Big]\nonumber\\
&\leq&
\frac{1}{2}E\Big[\sup_{t\in[0,T]}\|Z^{\e}(t)\|_\mathbb{W}^2\Big]
+\frac{C\e}{a^2(\e)}E\Big[\Big(\int_0^T\int_{\mathbb{Z}}M^2_G(z)\big(1+\|X^\e(s)\|_{\mathbb{V}}^2\big)\varphi_\e(s,z)\nu(dz)ds\Big)\Big]\nonumber\\
&\leq&
\frac{1}{2}E\Big[\sup_{t\in[0,T]}\|Z^{\e}(t)\|_\mathbb{W}^2\Big]
+\frac{C\e}{a^2(\e)}\varsigma_{M_G}\big(a^2(\e)+T\big)\Big(E\Big[\sup_{t\in[0,T]}\|X^{\e}(t)\|_\mathbb{V}^2\Big]+1\Big)\nonumber\\
&\leq&
\frac{1}{2}E\Big[\sup_{t\in[0,T]}\|Z^{\e}(t)\|_\mathbb{W}^2\Big]+\frac{C\e}{a^2(\e)}\varsigma_{M_G}\big(a^2(\e)+T\big).
\end{eqnarray}
And similarly
\begin{eqnarray}\label{MDP-2-eq-1-6}
& &E\Big[\frac{\e^2}{a^2(\e)}\int_0^t\int_{\mathbb{Z}}\|\widehat{G}(X^\e(s-),z)\|^2_{\mathbb{W}}N^{\e^{-1}\varphi^\e}(dzds)\Big]\nonumber\\
&=&
\frac{\e}{a^2(\e)}E\Big[\int_0^T\int_{\mathbb{Z}}\|\widehat{G}(X^\e(s),z)\|^2_{\mathbb{W}}\varphi^\e(s,z)(dz)ds\Big]\nonumber\\
&\leq&
\frac{C\e}{a^2(\e)}\varsigma_{M_G}\big(a^2(\e)+T\big).
\end{eqnarray}
Combining (\ref{MDP-2-eq-1-3})-(\ref{MDP-2-eq-1-6}), (\ref{condition}) and applying Gronwall's inequality, we obtain
\begin{eqnarray}\label{MDP-2-eq-1-estimate}
\lim_{\e\rightarrow0}E\Big[\sup_{t\in[0,T]}\|Z^\e(t)\|_{\mathbb{W}}^2\Big]=0.
\end{eqnarray}

\indent \textbf{Step 2.} Recall $\psi^\e=(\varphi^\e-1)/a(\e)$. Let $L^\e(t)$ be the unique solution of
\begin{eqnarray}\label{MDP-2-eq-2-1}
dL^\e(t)=-\kappa\widehat{A}L^\e(t)dt+\int_{\mathbb{Z}}\widehat{G}(X^\e(t),z)\psi^\e(z,t)1_{\{|\psi^\e|>\beta/a(\e)\}}\nu(dz)dt,
\end{eqnarray}
with initial value $L^\e(0)=0$. Using similar arguments as getting (\ref{MDP-2-eq-1-3}), we have
\begin{eqnarray}\label{MDP-2-eq-2-3}
& &\|L^\e(t)\|_{\mathbb{W}}^2+\frac{2\kappa}{\alpha}\int_0^t\|L^\e(s)\|_{\mathbb{W}}^2ds\nonumber\\
&=&
\frac{2\kappa}{\alpha}\int_0^t\Big(curl\big(L^\e(s)\big),curl\big(L^\e(s)-\alpha\Delta L^\e(s)\big)\Big)ds\nonumber\\
& &\qquad
+2\int_0^t\int_{\mathbb{Z}}\big(\widehat{G}(X^\e(s),z)\psi^\e(z,s)1_{\{|\psi^\e|>\beta/a(\e)\}},L^\e(s)\big)_{\mathbb{W}}\nu(dz)ds\nonumber\\
&\leq&
C\int_0^t\|L^\e(s)\|_{\mathbb{W}}^2ds+2\int_0^T\int_{\mathbb{Z}}
\|\widehat{G}(X^\e(s),z)\|_{\mathbb{W}}|\psi^\e(z,s)|1_{\{|\psi^\e|>\beta/a(\e)\}}\|L^\e(s)\|_{\mathbb{W}}\nu(dz)ds\nonumber\\
&\leq&
C\int_0^t\|L^\e(s)\|_{\mathbb{W}}^2ds\nonumber\\
& &\qquad
+C\sup_{t\in[0,T]}\|L^\e(t)\|_{\mathbb{W}}\sup_{t\in[0,T]}\big(1+\|X^\e(t)\|_{\mathbb{W}}\big)
\int_0^T\int_{\mathbb{Z}}M_G(z)|\psi^\e(z,s)|1_{\{|\psi^\e|>\beta/a(\e)\}}\nu(dz)ds\nonumber\\
&\leq&
C\int_0^t\|L^\e(s)\|_{\mathbb{W}}^2ds+\frac{1}{2}\sup_{t\in[0,T]}\|L^\e(t)\|^2_{\mathbb{W}}\nonumber\\
& &\qquad
+C\sup_{t\in[0,T]}\big(1+\|X^\e(t)\|^2_{\mathbb{W}}\big)
\Big\{\int_0^T\int_{\mathbb{Z}}M_G(z)|\psi^\e(z,s)|1_{\{|\psi^\e|>\beta/a(\e)\}}\nu(dz)ds\Big\}^2.
\end{eqnarray}
Noticing that, by (\ref{MDP-2-estimate-1}) and Lemma \ref{H-lemma-3}, we have
\begin{eqnarray}\label{MDP-2-eq-2-4}
CE\Big[\sup_{t\in[0,T]}\big(1+\|X^\e(t)\|^2_{\mathbb{W}}\big)\Big]
\Big\{\sup_{\psi\in{{\mathcal{S}}^M_\e}}\int_0^T\int_{\mathbb{Z}}M_G(z)|\psi(s,z)|1_{\{|\psi|>\beta/a(\e)\}}\nu(dz)ds\Big\}^2\rightarrow0\text{ as }\e\rightarrow0.
\end{eqnarray}
Combining (\ref{MDP-2-eq-2-3}) and (\ref{MDP-2-eq-2-4}) and applying Gronwall's inequality, we obtian
\begin{eqnarray}\label{MDP-2-eq-2-estimate}
\lim_{\e\rightarrow0}E\Big[\sup_{t\in[0,T]}\|L^\e(t)\|_{\mathbb{W}}^2\Big]=0.
\end{eqnarray}

\indent \textbf{Step 3.} Denote $U^\e$ the solution of the following equation
\begin{eqnarray}\label{MDP-2-eq-3-1}
dU^\e(t)=-\kappa\widehat{A}U^\e(t)dt+\int_{\mathbb{Z}}\big(\widehat{G}(X^\e(t),z)-\widehat{G}(u^0(t),z)\big)\psi^\e(z,t)1_{\{|\psi^\e|\leq\beta/a(\e)\}}\nu(dz)dt.
\end{eqnarray}
Similar to {\bf step 2}, we have
\begin{eqnarray}\label{MDP-2-eq-3-3}
& &\|U^\e(t)\|_{\mathbb{W}}^2+\frac{2\kappa}{\alpha}\int_0^t\|U^\e(s)\|_{\mathbb{W}}^2ds\nonumber\\
&=&
\frac{2\kappa}{\alpha}\int_0^t\Big(curl\big(U^\e(s)\big),curl\big(U^\e(s)-\alpha\Delta U^\e(s)\big)\Big)ds\nonumber\\
& &\qquad
+2\int_0^t\int_{\mathbb{Z}}\Big(\big(\widehat{G}(X^\e(s),z)-\widehat{G}(u^0(s),z)\big)
\psi^\e(z,t)1_{\{|\psi^\e|\leq\beta/a(\e)\}},U^\e(s)\Big)_{\mathbb{W}}\nu(dz)ds\nonumber\\
&\leq&
C\int_0^t\|U^\e(s)\|_{\mathbb{W}}^2ds+2\int_0^T\int_{\mathbb{Z}}
\|\widehat{G}(X^\e(s),z)-\widehat{G}(u^0(s),z)\|_{\mathbb{W}}|\psi^\e(z,s)|1_{\{|\psi^\e|\leq\beta/a(\e)\}}\|U^\e(s)\|_{\mathbb{W}}\nu(dz)ds\nonumber\\
&\leq&
C\int_0^t\|U^\e(s)\|_{\mathbb{W}}^2ds\nonumber\\
& &\qquad
+C\sup_{t\in[0,T]}\|U^\e(t)\|_{\mathbb{W}}\sup_{t\in[0,T]}\|X^\e(t)-u^0(t)\|_{\mathbb{V}}
\int_0^T\int_{\mathbb{Z}}L_G(z)|\psi^\e(z,s)|\nu(dz)ds\nonumber\\
&\leq&
C\int_0^t\|U^\e(s)\|_{\mathbb{W}}^2ds+\frac{1}{2}\sup_{t\in[0,T]}\|U^\e(t)\|^2_{\mathbb{W}}\nonumber\\
& &\qquad
+C\sup_{t\in[0,T]}\|X^\e(t)-u^0(t)\|^2_{\mathbb{V}}
\sup_{\psi\in\mathcal{S}_{\e}^M}\Big\{\int_0^T\int_{\mathbb{Z}}L_G(z)|\psi(z,s)|\nu(dz)ds\Big\}^2
\end{eqnarray}
Noticing that, by (\ref{MDP-2-theorem-1-estimate}) and Lemma \ref{H-lemma-2}, we have
\begin{eqnarray}\label{MDP-2-eq-3-4}
CE\Big[\sup_{t\in[0,T]}\|X^\e(t)-u^0(t)\|^2_{\mathbb{V}}
\sup_{\psi\in\mathcal{S}_{\e}^M}\Big\{\int_0^T\int_{\mathbb{Z}}L_G(z)|\psi(z,s)|\nu(dz)ds\Big\}^2\Big]\rightarrow0\text{ as }\e\rightarrow0.
\end{eqnarray}
Combining (\ref{MDP-2-eq-3-3}) and (\ref{MDP-2-eq-3-4}) and applying Gronwall's inequality, we obtain
\begin{eqnarray}\label{MDP-2-eq-3-estimate}
\lim_{\e\rightarrow0}E\Big[\sup_{t\in[0,T]}\|U^\e(t)\|_{\mathbb{W}}^2\Big]=0.
\end{eqnarray}

\indent \textbf{Step 4.} Set $K^\e=Z^\e+L^\e+U^\e$ and denote $\Upsilon^\e=Y^\e-K^\e$. By (\ref{4.1}), we have
\begin{eqnarray}\label{MDP-2-eq-4-1}
d\Upsilon^\e(t)
&=&
-\widehat{A}\Upsilon^\e(t)dt-a(\e)\widehat{B}(\Upsilon^\e(t)+K^\e(t),\Upsilon^\e(t)+K^\e(t))dt\nonumber\\
& &-\widehat{B}(u^0(t),\Upsilon^\e(t)+K^\e(t))dt-\widehat{B}(\Upsilon^\e(t)+K^\e(t),u^0(t))dt\nonumber\\
& &+\frac{1}{a(\e)}\Big(\widehat{F}\big(u^0(t)+a(\e)\big(\Upsilon^\e(t)+K^\e(t)\big),t\big)-\widehat{F}(u^0(t),t)\Big)dt\nonumber\\
& &+\int_{\mathbb{Z}}\widehat{G}(u^0(t),z)\psi^\e(z,t)1_{\{|\psi^\e|\leq\beta/a(\e)\}}\nu(dz)dt.
\end{eqnarray}
\indent Set
$$\Pi=\Big(D([0,T],\mathbb{V})\cap L^2([0,T],\mathbb{W});C([0,T],\mathbb{V})\cap L^2([0,T],\mathbb{W});B_2\big(\sqrt{M\kappa_2(1)}\big)\Big).$$
\indent By (\ref{MDP-2-eq-1-estimate}), (\ref{MDP-2-eq-2-estimate}), (\ref{MDP-2-eq-3-estimate}), and notice that $(\psi^\e1_{\{|\psi^\e|\leq\beta/a(\e)\}})_{\e>0}$ is tight in $B_2\big(\sqrt{M\kappa_2(1)}\big)$ with the weak topology of $L^2(\nu_T)$ (see Lemma 3.2 in \cite{BDG}), $(Z^\e,L^\e+U^\e,\psi^\e1_{\{|\psi^\e|\leq\beta/a(\e)\}})_{\e>0}$ is tight in $\Pi$, and let $(0,0,\psi)$ be any limit point of the tight family, and denote by $Y=\mathcal{G}^0(\psi)$ the solution of equation (\ref{Theorem-1-eq-1}).\\
\indent It follows from the Skorokhod representation theorem that there exist a probability space $(\Omega^1,\mathcal{F}^1,\mathbb{P}^1)$ and on this space, $\Pi$-valued random variables $(\widetilde{Z}^\e,\widetilde{J}^\e,\widetilde{\psi}^\e),(0,0,\widetilde{\psi}),\e\in(0,\e_0)$, such that $(\widetilde{Z}^\e,\widetilde{J}^\e,\widetilde{\psi}^\e)$ \big(respectively $(0,0,\widetilde{\psi})$\big) has the same law as $(Z^\e,L^\e+U^\e,\psi^\e1_{\{|\psi^\e|\leq\beta/a(\e)\}})$ \big(respectively $(0,0,\psi)$\big), and $(\widetilde{Z}^\e,\widetilde{J}^\e,\widetilde{\psi}^\e)\rightarrow(0,0,\widetilde{\psi})$ in $\Pi$, $\mathbb{P}^1$-a.s..\\
\indent Set $\widetilde{K}^\e=\widetilde{Z}^\e+\widetilde{J}^\e$. Denote by $\widetilde{\Upsilon}^\e(t)$ the unique solution of (\ref{MDP-2-eq-4-1}) with $(K^\e,\psi^\e1_{\{|\psi^\e|\geq\beta/a(\e)\}})$ replaced by $(\widetilde{K}^\e,\widetilde{\psi}^\e)$. Then $(\widetilde{K}^\e,\widetilde{\Upsilon}^\e)$ has the same law as $({K}^\e,{\Upsilon}^\e)$. Hence, $\widetilde{Y}^\e=\widetilde{K}^\e+\widetilde{\Upsilon}^\e$ has the same law as ${Y}^\e={K}^\e+{\Upsilon}^\e$ in $D([0,T],\mathbb{V})\cap L^2([0,T],\mathbb{W})$. Denote by $\widetilde{Y}$ the solution of equation (\ref{Theorem-1-eq-1}) with $\psi(z,t)$ replaced by $\widetilde{\psi}(z,t)$. $\widetilde{Y}$ must have the same law as $Y$.\\
\indent Thus the proof of the proposition will be complete if we can show that
\begin{eqnarray}\label{MDP-2-eq}
\sup_{t\in[0,T]}\|\widetilde{Y}^\e(t)-\widetilde{Y}(t)\|_{\mathbb{V}}\rightarrow0,\ \ \mathbb{P}^1-a.s., \text{ as }\e\rightarrow0.
\end{eqnarray}
Consider the following equation
\begin{eqnarray}\label{MDP-2-eq-5-1}
d\widetilde{\Gamma}^\e(t)=-\widehat{A}\widetilde{\Gamma}^\e(t)dt+\int_{\mathbb{Z}}\widehat{G}(u^0(t),z)\widetilde{\psi}^\e(z,t)\nu(dz)dt,
\end{eqnarray}
and
\begin{eqnarray}\label{MDP-2-eq-5-2}
d\widetilde{\Gamma}(t)=-\widehat{A}\widetilde{\Gamma}(t)dt+\int_{\mathbb{Z}}\widehat{G}(u^0(t),z)\widetilde{\psi}(z,t)\nu(dz)dt.
\end{eqnarray}
As the proof of (\ref{MDP-1-estimate-3}), first we can show
\begin{eqnarray}\label{MDP-2-eq-5-estimate}
\lim_{\e\rightarrow0}\sup_{t\in[0,T]}\|\widetilde{\Gamma}^\e(t)-\widetilde{\Gamma}(t)\|_{\mathbb{V}}^2=0,\ \ \mathbb{P}^1-\text{a.s.}.
\end{eqnarray}

Set $\widetilde{M}=\widetilde{Y}-\widetilde{\Gamma}$ and $\widetilde{M}^\e=\widetilde{Y}^\e-\widetilde{K}^\e-\widetilde{\Gamma}^\e$. Then
\begin{eqnarray}\label{MDP-2-eq-6-1}
d\widetilde{M}(t)
&=&
-\kappa\widehat{A}\widetilde{M}(t)dt-\widehat{B}\big(\widetilde{M}(t)+\widetilde{\Gamma}(t),u^0(t)\big)dt
-\widehat{B}\big(u^0(t),\widetilde{M}(t)+\widetilde{\Gamma}(t)\big)dt\nonumber\\
& &
+\widehat{F}'(u^0(t),t)\big(\widetilde{M}(t)+\widetilde{\Gamma}(t)\big)dt.
\end{eqnarray}
and
\begin{eqnarray}\label{MDP-2-eq-6-2}
d\widetilde{M}^\e(t)
&=&
-\kappa\widehat{A}\widetilde{M}^\e(t)dt
-a(\e)\widehat{B}\big(\widetilde{M}^\e(t)+\widetilde{K}^\e(t)+\widetilde{\Gamma}^\e(t),
\widetilde{M}^\e(t)+\widetilde{K}^\e(t)+\widetilde{\Gamma}^\e(t)\big)dt\nonumber\\
& &
-\widehat{B}\big(\widetilde{M}^\e(t)+\widetilde{K}^\e(t)+\widetilde{\Gamma}^\e(t),u^0(t)\big)dt
-\widehat{B}\big(u^0(t),\widetilde{M}^\e(t)+\widetilde{K}^\e(t)+\widetilde{\Gamma}^\e(t)\big)dt\nonumber\\
& &
+\frac{1}{a(\e)}\Big(\widehat{F}\big(u^0(t)+a(\e)\big(\widetilde{M}^\e(t)+\widetilde{K}^\e(t)+\widetilde{\Gamma}^\e(t)\big),t\big)
-\widehat{F}\big(u^0(t),t\big)\Big)dt.
\end{eqnarray}
Since
\begin{eqnarray}\label{MDP-2-eq-6-3}
\lim_{\e\rightarrow0}\sup_{t\in[0,T]}\|\widetilde{K}^\e(t)\|_{\mathbb{W}}^2=0,\ \ \mathbb{P}^1-\text{a.s.},
\end{eqnarray}
taking into account (\ref{MDP-2-eq-5-estimate}), the proof of (\ref{MDP-2-eq}) reduces to show
\begin{eqnarray}\label{MDP-2-eq-7-estimate}
\lim_{\e\rightarrow0}\sup_{t\in[0,T]}\|\widetilde{M}^\e(t)-\widetilde{M}(t)\|_{\mathbb{V}}=0,\ \ \mathbb{P}^1-a.s.
\end{eqnarray}
By the similar arguments as in the proof of (\ref{MDP-1-estimate-1}) and using (\ref{MDP-2-eq-5-estimate}) and (\ref{MDP-2-eq-6-3}) again, we have
\begin{eqnarray}\label{MDP-2-eq-6-4}
\sup_{\e\in(0,\e_0)}\Big[\sup_{t\in[0,T]}\|\widetilde{M}^\e(t)\|_{\mathbb{W}}^2+\sup_{t\in[0,T]}\|\widetilde{M}(t)\|_{\mathbb{W}}^2\Big]\leq C(\omega^1)<\infty,\ \ \mathbb{P}^1-\text{a.s.}.
\end{eqnarray}

Set $\widetilde{N}^\e=\widetilde{M}^\e-\widetilde{M}$ and $\widetilde{H}^\e=\widetilde{\Gamma}^\e-\widetilde{\Gamma}$, we have
\begin{eqnarray}\label{MDP-2-eq-7-1}
d\widetilde{N}^\e(t)
&=&
-\kappa\widehat{A}\widetilde{N}^\e(t)dt
-a(\e)\widehat{B}\big(\widetilde{M}^\e(t)+\widetilde{K}^\e(t)+\widetilde{\Gamma}^\e(t),
\widetilde{M}^\e(t)+\widetilde{K}^\e(t)+\widetilde{\Gamma}^\e(t)\big)dt\nonumber\\
& &
-\widehat{B}\big(\widetilde{N}^\e(t)+\widetilde{K}^\e(t)+\widetilde{H}^\e(t),u^0(t)\big)dt
-\widehat{B}\big(u^0(t),\widetilde{N}^\e(t)+\widetilde{K}^\e(t)+\widetilde{H}^\e(t)\big)dt\nonumber\\
& &
+\frac{1}{a(\e)}\Big(\widehat{F}\big(u^0(t)+a(\e)\big(\widetilde{M}^\e(t)+\widetilde{K}^\e(t)+\widetilde{\Gamma}^\e(t)\big),t\big)
-\widehat{F}\big(u^0(t),t\big)\Big)dt\nonumber\\
& &
-\widehat{F}'(u^0(t),t)\big(\widetilde{M}(t)+\widetilde{\Gamma}(t)\big)dt.
\end{eqnarray}
\indent Applying the chain rule, we have
\begin{eqnarray}\label{MDP-2-eq-7-2}
& &\|\widetilde{N}^\e(t)\|_{\mathbb{V}}^2+2\kappa\int_0^t\|\widetilde{N}^\e(s)\|^2ds\nonumber\\
&=&
-2a(\e)\int_0^t\Big(\widehat{B}\big(\widetilde{M}^\e(s)+\widetilde{K}^\e(s)+\widetilde{\Gamma}^\e(s),
\widetilde{M}^\e(s)+\widetilde{K}^\e(s)+\widetilde{\Gamma}^\e(s)\big),\widetilde{N}^\e(s)\Big)_{{\mathbb{W}}^*,\mathbb{W}}ds\nonumber\\
& &
-2\int_0^t\Big(\widehat{B}\big(\widetilde{N}^\e(s)+\widetilde{K}^\e(s)+\widetilde{H}^\e(s),u^0(s)\big),
\widetilde{N}^\e(s)\Big)_{{\mathbb{W}}^*,\mathbb{W}}ds\nonumber\\
& &
-2\int_0^t\Big(\widehat{B}\big(u^0(s),\widetilde{N}^\e(s)+\widetilde{K}^\e(s)+\widetilde{H}^\e(s)\big),
\widetilde{N}^\e(s)\Big)_{{\mathbb{W}}^*,\mathbb{W}}ds\nonumber\\
& &
+\frac{2}{a(\e)}\int_0^t\Big(\widehat{F}\big(u^0(s)+a(\e)\big(\widetilde{M}^\e(s)+\widetilde{K}^\e(s)+\widetilde{\Gamma}^\e(s)\big),s\big)
-\widehat{F}\big(u^0(s),s\big),\widetilde{N}^\e(s)\Big)_{\mathbb{V}}ds\nonumber\\
& &
-2\int_0^t\Big(\widehat{F}'(u^0(s),s)\big(\widetilde{M}(s)+\widetilde{\Gamma}(s)\big),\widetilde{N}^\e(s)\Big)_{\mathbb{V}}ds\nonumber\\
&=&
I_1(t)+I_2(t)+I_3(t)+I_4(t).
\end{eqnarray}
Now we estimate $I_i, i=1,2,3,4$ respectively. By Lemma \ref{Lem-B-01}, Lemma \ref{Regularity} and (\ref{MDP-2-eq-6-4}), we have
\begin{eqnarray}\label{MDP-2-eq-7-3}
|I_1(t)|
&\leq&
2a(\e)\int_0^t\|\widetilde{M}^\e(s)+\widetilde{K}^\e(s)+\widetilde{\Gamma}^\e(s)\|_{\mathbb{V}}^2\|\widetilde{N}^\e(s)\|_{\mathbb{W}}ds\nonumber\\
&\leq&
4a(\e)C(\omega_1,T)\sup_{s\in[0,T]}
\Big(\|\widetilde{M}^\e(s)\|^2_{\mathbb{V}}+\|\widetilde{K}^\e(s)\|^2_{\mathbb{V}}+\|\widetilde{\Gamma}^\e(s)\|^2_{\mathbb{V}}\Big),\nonumber\\
|I_2(t)|
&=&
2\Big|\int_0^t\Big(\widehat{B}\big(\widetilde{N}^\e(s)+\widetilde{K}^\e(s)+\widetilde{H}^\e(s),\widetilde{N}^\e(s)\big),
u^0(s)\Big)_{{\mathbb{W}}^*,\mathbb{W}}ds\Big|\nonumber\\
&\leq&
2\Big|\int_0^t\Big(\widehat{B}\big(\widetilde{N}^\e(s)+\widetilde{K}^\e(s)+\widetilde{H}^\e(s),\widetilde{N}^\e(s)+\widetilde{K}^\e(s)+\widetilde{H}^\e(s)\big),
u^0(s)\Big)_{{\mathbb{W}}^*,\mathbb{W}}ds\Big|\nonumber\\
& &\qquad
+2\Big|\int_0^t\Big(\widehat{B}\big(\widetilde{K}^\e(s)+\widetilde{H}^\e(s),\widetilde{K}^\e(s)+\widetilde{H}^\e(s)\big),
u^0(s)\Big)_{{\mathbb{W}}^*,\mathbb{W}}ds\Big|\nonumber\\
& &\qquad
+2\Big|\int_0^t\Big(\widehat{B}\big(\widetilde{N}^\e(s),\widetilde{K}^\e(s)+\widetilde{H}^\e(s)\big),
u^0(s)\Big)_{{\mathbb{W}}^*,\mathbb{W}}ds\Big|\nonumber\\
&\leq&
C\int_0^t\|\widetilde{N}^\e(s)\|_{\mathbb{V}}^2ds
+C\sup_{s\in[0,T]}\big(\|\widetilde{K}^\e(s)\|^2_{\mathbb{V}}+\|\widetilde{H}^\e(s)\|^2_{\mathbb{V}}\big)
+C(\omega_1,T)\sup_{s\in[0,T]}\big(\|\widetilde{K}^\e(s)\|_{\mathbb{V}}+\|\widetilde{H}^\e(s)\|_{\mathbb{V}}\big),\nonumber\\
\text{and}\nonumber\\
|I_3(t)|
&=&
2\int_0^t\Big(\widehat{B}\big(u^0(s),\widetilde{K}^\e(s)+\widetilde{H}^\e(s)\big),
\widetilde{N}^\e(s)\Big)_{{\mathbb{W}}^*,\mathbb{W}}ds\nonumber\\
&\leq&
C(\omega_1,T)\sup_{s\in[0,T]}\big(\|\widetilde{K}^\e(s)\|_{\mathbb{V}}+\|\widetilde{H}^\e(s)\|_{\mathbb{V}}\big).\nonumber
\end{eqnarray}
Noticing the fact: there exists $\theta(s)\in[0,1]$ depending on $s,x,y$, such that
\begin{eqnarray*}
\widehat{F}(x+y,s)-\widehat{F}(x,s)=\widehat{F}'(x+\theta y,s)y,\text{ for any } x,y\in\mathbb{V}.
\end{eqnarray*}
Combining Conditions {\bf (F1)} and {\bf (F2)}, we have
\begin{eqnarray}\label{MDP-2-eq-7-6}
|I_4(t)|
&\leq&
2\int_0^t\Big|\Big(\frac{1}{a(\e)}\Big(\widehat{F}\big(u^0(s)+a(\e)\big(\widetilde{M}^\e(s)+\widetilde{K}^\e(s)+\widetilde{\Gamma}^\e(s)\big),s\big)
-\widehat{F}\big(u^0(s),s\big)\Big)-\nonumber\\
& &\qquad\qquad\widehat{F}'(u^0(s),s)\big(\widetilde{M}^\e(s)+\widetilde{K}^\e(s)+\widetilde{\Gamma}^\e(s)\big),
\widetilde{N}^\e(s)\Big)_{\mathbb{V}}\Big|ds\nonumber\\
& &
+2\int_0^t\Big|\Big(\widehat{F}'(u^0(s),s)\big(\widetilde{M}^\e(s)+\widetilde{K}^\e(s)+\widetilde{\Gamma}^\e(s)\big)-
\widehat{F}'(u^0(s),s)\big(\widetilde{M}(s)+\widetilde{\Gamma}(s)\big),\widetilde{N}^\e(s)\Big)_{\mathbb{V}}\Big|ds\nonumber\\
&\leq&
C(\omega_1,T)a(\e)\sup_{s\in[0,T]}\big(\|\widetilde{M}^\e(s)\|^2_{\mathbb{V}}+
\|\widetilde{K}^\e(s)\|^2_{\mathbb{V}}+\|\widetilde{\Gamma}^\e(s)\|^2_{\mathbb{V}}\big)\nonumber\\
& &+C\int_0^t\|\widetilde{N}^\e(s)\|^2_{\mathbb{V}}ds
+C(\omega_1,T)\sup_{s\in[0,T]}\big(\|\widetilde{K}^\e(s)\|^2_{\mathbb{V}}+\|\widetilde{H}^\e(s)\|^2_{\mathbb{V}}\big).\nonumber
\end{eqnarray}
Since $\lim_{\e\rightarrow0}a(\e)=0$ and
\begin{eqnarray*}
\lim_{\e\rightarrow0}\sup_{s\in[0,T]}\big(\|\widetilde{K}^\e(s)\|_{\mathbb{W}}+\|\widetilde{H}^\e(s)\|_{\mathbb{W}}\big)=0,\ \ \mathbb{P}^1-\text{a.s.},
\end{eqnarray*}
by Gronwall's inequality we obtain (\ref{MDP-2-eq-7-estimate}). The proof is completed.

\end{proof}

\def\refname{ References}


\begin{thebibliography}{2}

\bibitem{Acosta} A. De Acosta: {Moderate deviations and associated Laplace approximations for sums of
independent random vectors.} Trans. Amer. Math. Soc.  \textbf{329}(1), 357--375(1992)

\bibitem{BDG} A. Budhiraja, P. Dupuis and A. Ganguly: {Moderate Deviation Principle for Stochastic Differential Equations with jump.} available at arXiv:1401.7316

\bibitem{BDG2} A. Budhiraja, P. Dupuis and A. Ganguly: {Moderate Deviation Principle for Stochastic Differential Equations with jump.} Ann. Probab. \textbf{44}(4), 1723-1775(2016)

\bibitem{BDM}
A. Budhiraja, P. Dupuis and V. Maroulas: Variational representations for continuous time processes.
\newblock {\em Ann. Inst. Henri Poincar$\acute{e}$ Probab. Stat.}, \textbf{47}, 725-747(2011)

\bibitem{2003-Busuioc-p1119-1119}
A.~V. Busuioc and T.~S. Ratiu.
\newblock The second grade fluid and averaged Euler equations with Navier-slip boundary conditions.
\newblock {\em Nonlinearity}, 16(3):1119, 2003.

\bibitem{Chen} X. Chen: {Probabilities of moderate deviations for independent random vectors in a Banach space.}
Chinese J. Appl. Prob. Stat£¬ \textbf{7}, 24--32(1991)

\bibitem{CC}
N. Chemetov and F. Cipriano:
\newblock Well-posedness of stochastic second grade fluids.
\newblock {\em  J. Math. Anal. Appl.}, 454(2):585--616(2017)

\bibitem{CG} D. Cioranescu, V. Girault:  {Weak and classical solutions of a family of second grade fluids.}
Internat. J. Non-Linear Mech. \textbf{32}(2), 317-335(1997)

\bibitem{CO} D. Cioranescu, E.H. Ouazar:  {Existence and uniqueness for fluids of second grade.} In Nonlinear Partial Differential
Equations and Their Applications, College de France Seminar, vol. 109, pp. 178-197, Pitman, Boston, Mass, USA, 1984

\bibitem{DXZZ} Z. Dong, J. Xiong, J.L. Zhai, T.S. Zhang:{A Moderate Deviation Principle for 2-D Stochastic Navier-Stokes Equations Driven by Multiplicative L\'evy Noises.} J. Funct. Anal. \textbf{272}(1), 227-254(2017)

\bibitem{1974-Dunn-p191-252}
J.~E. Dunn and R.~L. Fosdick.
\newblock Thermodynamics, stability, and boundedness of fluids of complexity 2
  and fluids of second grade.
\newblock {\em Archive for Rational Mechanics and Analysis}, 56(3):191--252,
  1974.


\bibitem{Ermakov} M. Ermakov: {The sharp lower bound of asymptotic efficiency of estimators in the
zone of moderate deviation probabilities.} Electron. J. Statist.  \textbf{6}, 2150--2184(2012)

\bibitem{1979-Fosdick-p145-152}
R.~Fosdick and K.~Rajagopal.
\newblock Anomalous features in the model of ¡°second order fluids¡±.
\newblock {\em Archive for Rational Mechanics and Analysis}, 70(2):145--152,
  1979.

\bibitem{Gao} F. Gao, X. Zhao: {Delta method in largedeviations and moderate deviations for estimators.} Ann. Statist. \textbf{39}(2), 1211--1240(2011)

\bibitem{Guillin} A. Guillin, R. Liptser: {Examples of moderate deviation principle for diffusion processes.}
Discrete Contin. Dyn. Syst. Ser. B  \textbf{6}(2004)

\bibitem{HRS} E. Hausenblas, P.A. Razafimandimby, M. Sango:  {Martingale solution to equations for differential type fluids of grade two driven by random force of Levy type}. Potential Anal. \textbf{38}(4), 1291-1331(2013)

\bibitem{Ikeda-Watanabe}
N.~Ikeda and S.~Watanabe: {\em Stochastic differential equations and diffusion processes.} {\em North-Holland Mathematical Libaray} \textbf{24}, Amsterdam: North-Holland.(1981)

\bibitem{Inglot} T. Inglot, W. Kallenberg: {Moderate deviations of minimum contrast estimators under contamination.} Ann. Statist. 852--879(2003)

\bibitem{Kallenberg} W. Kallenberg: {On moderate deviation theory in estimation.}  Ann. Statist.  \textbf{11}(2), 498--504(1983)

\bibitem{Ledoux} M. Ledoux: {Sur les d{\'e}viations mod{\'e}r{\'e}es des sommes de variables al{\'e}atoires vectorielles ind{\'e}pendantes de m{\^e}me loi.} Ann. Henri Poincar\'{e} \textbf{28}(7), 267--280(1992)

\bibitem{Lions}
J.L. Lions: {\em Quelques M$\acute{e}$thodes de R$\acute{e}$solution des Probl$\grave{e}$mes aux Limites Non Lin$\acute{e}$aires.} Dunond, Paris(1969)

\bibitem{RS-10-01}
P.A. Razafimandimby, M. Sango:  {Asymptotic behaviour of solutions of stochastic evolution equations for second
grade fluids.} C.R. Acad. Sci. Paris \textbf{348}(13-14), 787-790(2010)

\bibitem{RS-10}
P.~A. Razafimandimby and M.~Sango.
\newblock Weak solutions of a stochastic model for two-dimensional second grade
  fluids.
\newblock {\em Boundary Value Problems}, 2010(1):1--47, 2010.

\bibitem{RS-12}
P.~A. Razafimandimby and M.~Sango.
\newblock Strong solution for a stochastic model of two-dimensional second
  grade fluids: Existence, uniqueness and asymptotic behavior.
\newblock {\em Nonlinear Analysis: Theory, Methods \& Applications},
  75(11):4251 -- 4270, 2012.

\bibitem{SZZ}
S.J. Shang, J.L. Zhai and T.S. Zhang:
\newblock Strong solutions to stochastic equations of second
  grade fluids driven by L$\acute{e}$vy processes.
Aviailable at  arXiv:1701.00314v1(2017)

\bibitem{2001-Shkoller-p539-543}
S.~Shkoller.
\newblock Smooth global Lagrangian flow for the 2D Euler and second-grade fluid equations.
\newblock {\em Applied Mathematics Letters}, 14(5):539 -- 543, 2001.

\bibitem{Solonnikov}
V.~Solonnikov: {On general boundary problems for systems which are elliptic in the
  sense of {A}. {D}ouglis and {L}. {N}irenberg. {II}.}
\newblock {\em Proceedings of the Steklov Institute of Mathematics},
  \textbf{92}: 269--339(1968).

\bibitem{Temam 1983}
R. Temam: {\em Navier-Stokes Equations and Nonlinear Functional Analysis.} {\em Regional Conference Series in Applied Mathematics} \textbf{41}, Philadelphia(1983)



\bibitem{Wang1} R. Wang, J.L. Zhai, T.S. Zhang: {A moderate deviation principle for 2-D stochastic Navier--Stokes equations.}
J. Diff. Equat.  \textbf{258}(10), 3363--3390(2015)

\bibitem{WZZ}
R. Wang, J.L. Zhai, T.S. Zhang: {\em Exponential mixing for stochastic model of two-dimensional second grade fluids.} {\em Nonlinear Anal.} \textbf{132}:196-213(2016)

\bibitem{Wang2} R. Wang, T.S. Zhang: {Moderate deviations for stochastic reaction-diffusion equations with multiplicative noise.}
Potential Anal. \textbf{42}, 99--113(2015)

\bibitem{Wu} L. Wu: {Moderate deviations of dependent random variables related to CLT.} Ann. Probab. 420--445(1995)

\bibitem{ZZ} J.L. Zhai, T.S. Zhang: {Large deviations for stochastic models of two-dimensional second grade fluids.} Appl. Math. Optimiz. 1--28(2015)

\bibitem{ZZZ1} J.L. Zhai, T.S. Zhang and W.T. Zheng: {Moderate deviations for stochastic models of two-dimensional second grade fluids.} Stoch. Dyn. \textbf{18}(3), 1850026(1-46)(2018)

\bibitem{ZZZ2} J.L. Zhai, T.S. Zhang and W.T. Zheng: {Large deviations for stochastic models of two-dimensional second grade fluids driven by L\'evy noise} Aviailable at arXiv:1706.08862(2017)



\end{thebibliography}
\end{document}